\documentclass[10pt,reqno]{amsart}
\usepackage{color}
\usepackage{hyperref}
\hypersetup{
  citecolor = blue,
  colorlinks,
  linkcolor = blue,
  urlcolor = blue
}
\usepackage{amssymb}
\usepackage{amsfonts}
\usepackage{amsmath}
\usepackage{amsthm}
\newcommand{\N}{\mathbb{N}}
\newcommand{\dd}{\,\mathrm{d}}
\newcommand{\pl}{\mathop{p\operatorname{-lim}}}
\newcommand{\R}{\mathbb{R}}
\newcommand{\C}{\mathbb{C}}
\newcommand{\Z}{\mathbb{Z}}
\newcommand{\Q}{\mathbb{Q}}
\newcommand{\T}{\mathbb{T}}
\newcommand{\A}{\mathbb{A}}
\newcommand{\DT}{DT}
\newcommand{\resfunc}{admissible function}
\renewcommand{\emptyset}{\varnothing}

\newtheorem{theorem}{Theorem}[section]
\newtheorem{definition}[theorem]{Definition}
\newtheorem{lemma}[theorem]{Lemma}
\newtheorem{corollary}[theorem]{Corollary}
\newtheorem{proposition}[theorem]{Proposition}

\newtheorem{example}[theorem]{Example}

\newtheorem{remark}[theorem]{Remark}

\numberwithin{equation}{section}

\begin{document}

\title{Van der Corput's Difference Theorem: some modern developments}

\author{Vitaly Bergelson}
\address[Bergelson]{Department of Mathematics, The Ohio State University, Columbus, OH 43210, USA}
\email{vitaly@math.ohio-state.edu}
\author{Joel Moreira}
\address[Moreira]{Department of Mathematics, The Ohio State University, Columbus, OH 43210, USA}
\email{moreira@math.ohio-state.edu}


\date{\today}
\thanks{The first author gratefully acknowledges the support of the NSF under grants DMS-1162073 and DMS-1500575}

\begin{abstract}We discuss various forms of the classical van der Corput's difference theorem and explore applications to and connections with the theory of uniform distribution, ergodic theory, topological dynamics and combinatorics.
\end{abstract}
\maketitle

\section{Introduction}\label{sec_intro}
Van der Corput's classical Difference Theorem (\cite{vdCorput31}) is traditionally stated as follows:
\begin{itemize}
  \item[\DT:] Let $(x_n)_{n=1}^\infty$ be a sequence of real numbers. If for any $h\in\N=\{1,2,\dots\}$, the sequence $(x_{n+h}-x_n)_{n=1}^\infty$ is uniformly distributed $\bmod1$, then $(x_n)_{n=1}^\infty$ is uniformly distributed $\bmod1$.
\end{itemize}
Let $f(x)\in\R[x]$ and let $x_n=f(n)$, $n\in\N$.
Then for any $h\in\N$, the degree of the polynomial $f(n+h)-f(n)=x_{n+h}-x_n$ equals $\deg f-1$.
This rather trivial observation, together with \DT{}, allows one to obtain a streamlined inductive proof of Weyl's celebrated theorem, (\cite{Weyl16}), which states that for any polynomial $f\in\R[x]$ which has at least one coefficient, other than the constant term, irrational, the sequence $x_n=f(n)$, $n\in\N$, is uniformly distributed $\bmod\ 1$.
While the proof of \DT{} can be condensed to just a few lines (see the proof of Theorem \ref{thm_vdCoriginal} in the next section), van der Corput's Difference Theorem contains in embryonic form a powerful idea of Complexity Reduction.
In the subsequent sections we will provide numerous examples of applications of various generalizations of \DT.
These applications include results in the theory of uniform distribution, as well as some new multiple recurrence theorems in Ergodic Theory and Topological Dynamics.

The structure of the paper is as follows.
In Section \ref{sec_vdCtrick} we provide a short proof of \DT{} and also formulate and prove a general form of \DT{} that will be utilized in later sections.
Section \ref{sec_welldist} is devoted to various aspects of the phenomenon of well distribution.
Section \ref{sec_besicovitch} deals with uniform distribution along Besicovitch almost periodic sequences.
Among other things, we establish new uniform distribution results involving the sequence of squarefree numbers. 
In Section \ref{sec_othergroups} we provide some additional applications of the (generalized) \DT{}.
In particular, we establish a version of Weyl's equidistribution theorem in the group $\A/\Q$ (where $\A$ is the additive group of adeles).
We also prove a novel kind of ergodic theorem involving measure preserving actions of the additive and multiplicative groups of a countable field.
Finally, in sub-section \ref{sec_ractions} we prove a rather general `non-linear' mean ergodic theorem involving a family of commuting measure preserving $\R$-actions.
In the relatively short Section \ref{sec_katai} we formulate a theorem due to K\'atai which may be interpreted as a multiplicative version of \DT{} and demonstrate its power by providing an application involving the classical M\"obius function.
In Section \ref{sec_ultrafilter} we consider limits along idempotent ultrafilters and prove, with the help of an ultrafilter variant of \DT, a polynomial ergodic theorem for mildly mixing transformations.
Finally, in Section \ref{sec_topologicalvdC} we show that the idea of Complexity Reduction can be applied to multiple recurrence in topological dynamics, with new applications to combinatorics.

Regrettably, due to various constrains, numerous additional applications of, and connections with, \DT{} were not included in this paper. See, for example, \cite{Bergelson_Knutson09,Bergelson_Lesigne08,Bergelson_McCutcheon00,Bergelson_McCutcheon10,Cigler64,Elst_Muller15,Furstenberg_Katznelson85,Haland93,Kamae_France78,Kemperman64,Ruzsa84,Taschner79}.

\section{Some variants of van der Corput's Difference Theorem}\label{sec_vdCtrick}
We start with the classical definition of uniform distribution, which goes back to the groundbreaking paper of Weyl \cite{Weyl16}.
\begin{definition}\label{def_uniformdistribution}
A sequence $(x_n)_{n\in\N}$ taking values in the unit interval $[0,1)$ is \emph{uniformly distributed} if for any subinterval $[a,b)\subset[0,1)$ the proportion of those $n\in\N$ for which $x_n\in[a,b)$ is $b-a$. More precisely, if
\begin{equation}\label{eq_intro_uniformdistribution}
\lim_{N\to\infty}\frac{\Big|\big\{n\in\{1,\dots,N\}:x_n\in[a,b)\big\}\Big|}N=b-a
\end{equation}
for all $0\leq a<b\leq1$.
\end{definition}
One can informally say that a sequence is uniformly distributed if each interval gets its fair share of points.
Since there are uncountably many subintervals $[a,b)\subset[0,1)$, it is a priori not clear from the definition that such sequences exist (although one should certainly believe that a sequence of i.i.d.\ random variables taking values in $[0,1)$ and having probability density function $1$ will have this property almost surely.)

The following classical theorem, providing useful characterizations of uniform distribution, is originally due to Weyl \cite{Weyl16}. For a proof see, for instance, \cite[Theorems 1.1.1 and 1.2.1]{Kuipers_Niederreiter74}.
\begin{theorem}\label{theorem_weyl}
 Let $(x_n)_{n\in\N}$ be a sequence taking values in $[0,1)$. The following are equivalent:
 \begin{enumerate}
   \item $(x_n)_{n\in\N}$ is uniformly distributed.
   \item $$\lim_{N\to\infty}\frac1N\sum_{n=1}^Nf(x_n)=\int_0^1f(x)\dd x\qquad\qquad\forall f\in C[0,1]$$
   \item $$\lim_{N\to\infty}\frac1N\sum_{n=1}^Ne^{2\pi ihx_n}=0\qquad\qquad\forall h\in\N$$
 \end{enumerate}
\end{theorem}

Condition (3) of Theorem \ref{theorem_weyl} is known as the \emph{Weyl criterion} for uniform distribution.
Using it, one can easily show, for instance, that for any irrational $\alpha\in\R$, the sequence $x_n=n\alpha\bmod1$\footnote{For a real number $x$, we denote by $\lfloor x\rfloor$ the largest integer not exceeding $x$, and denote by $x\bmod1$ the number $x-\lfloor x\rfloor\in[0,1)$.} is uniformly distributed.

In this paper, as is customary in the theory of uniform distribution, we identify $[0,1)$ with the torus $\T:=\R/\Z$.
The van der Corput difference theorem, \DT, gives a sufficient condition for a sequence $(x_n)_{n\in\N}$ taking values in $\T$ to be uniformly distributed:
\begin{theorem}[van der Corput's difference theorem, \cite{vdCorput31}]\label{thm_vdCoriginal}
Let $(x_n)_{n\in\N}$ be a sequence taking values in the torus $\T$.
Assume that for every $d\in\N$, the sequence $n\mapsto x_{n+d}-x_n$ is uniformly distributed.
Then $(x_n)_{n\in\N}$ is uniformly distributed.
\end{theorem}

In view of Weyl's criterion, the assumption of Theorem \ref{thm_vdCoriginal} is that, for any $d,h\in\N$,
$$\lim_{N\to\infty}\frac1N\sum_{n=1}^Ne^{2\pi ih(x_{n+d}-x_n)}=0$$
 and to prove Theorem \ref{thm_vdCoriginal} we have to show that
 $$\lim_{N\to\infty}\frac1N\sum_{n=1}^Ne^{2\pi ihx_n}=0$$
 for any $h\in\N$.
Therefore, Theorem \ref{thm_vdCoriginal} is a corollary of the following version of \DT.
\begin{theorem}\label{thm_vdCinC}
  Let $(u_n)_{n\in\N}$ be a bounded sequence in $\C$.
  Assume that for every $d\in\N$ we have
  \begin{equation}\label{eq_thm_vdCinC}
    \lim_{N\to\infty}\frac1N\sum_{n=1}^Nu_{n+d}\overline{u_n}=0
  \end{equation}
  Then
  $$\lim_{N\to\infty}\frac1N\sum_{n=1}^Nu_n=0$$
\end{theorem}

\begin{proof}
  Our goal will be achieved if we show that for any $\epsilon>0$
  $$\limsup_{N\to\infty}\left|\frac1N\sum_{n=1}^Nu_n\right|<\epsilon$$
  Notice now that for any $\epsilon>0$ and any $D\in\N$, if $N\in\N$ is large enough we have
  $$\left|\frac1N\sum_{n=1}^Nu_n-\frac1N\frac1D\sum_{n=1}^N\sum_{d=1}^Du_{n+d}\right|<\frac\epsilon2$$
   Hence it suffices to show that, if $D$ is large enough,
   $$\limsup_{N\to\infty}\left|\frac1N\frac1D\sum_{n=1}^N\sum_{d=1}^Du_{n+d}\right|<\frac\epsilon2$$
  Using the Cauchy-Schwarz inequality in $\R^N$ we have
  \begin{eqnarray}\label{eq_vdcproof}
    \limsup_{N\to\infty}\left|\frac1N\frac1D\sum_{n=1}^N\sum_{d=1}^Du_{n+d}\right|^2&\leq& \limsup_{N\to\infty}\frac1N\sum_{n=1}^N\left|\frac1D\sum_{d=1}^Du_{n+d}\right|^2 \notag\\&=&\limsup_{N\to\infty}\frac1N\sum_{n=1}^N\frac1{D^2}\sum_{d_1,d_2=1}^D u_{n+d_1}\overline{u_{n+d_2}}\notag \\&\leq&\frac1{D^2}\sum_{d_1,d_2=1}^D\limsup_{N\to\infty}\frac1N\sum_{n=1}^Nu_{n+d_1}\overline{u_{n+d_2}}\qquad
  \end{eqnarray}
  Note that, for $d_1\neq d_2$, it follows from \eqref{eq_thm_vdCinC} that $\tfrac1N\sum_{n=1}^Nu_{n+d_1}\overline{u_{n+d_2}}\to0$ as $N\to\infty$.
  We conclude that the quantity in \eqref{eq_vdcproof} is bounded by $\tfrac D{D^2}=\tfrac1D$ which is arbitrarily small for large enough $D$.
\end{proof}
As was alluded to in the introduction, \DT{} allows one to easily show that for any polynomial $f\in\R[x]$ which has an irrational coefficient, other than the constant term, the sequence $\big(f(n)\bmod1\big)_{n\in\N}$ is uniformly distributed. 
We will describe now another immediate application of \DT.

A classical result in the theory of uniform distribution is Fej\'er's theorem, which we will presently formulate.
Given a sequence $n\mapsto f(n)$ of real numbers, we define its discrete derivative $\Delta f:\N\to\R$ by the formula $\Delta f(n)=f(n+1)-f(n)$.
We also denote by $\Delta^s f$ the iterated discrete derivative, i.e.\ $\Delta^sf=\Delta(\Delta^{s-1}f)$.
\begin{theorem}[Fej\'er; see, for example, {\cite[Theorem 1.2.5]{Kuipers_Niederreiter74}}]\label{theorem_fejer}
  Let $\big(f(n)\big)_{n\in\N}$ be sequence in $\R$.
  Assume that $\Delta f(n)$ is eventually decreasing and satisfies
  $$\lim_{n\to\infty}\Delta f(n)=0\qquad\text{and}\qquad\lim_{n\to\infty}n\Delta f(n)=\infty$$
  Then the sequence $\big(f(n)\bmod1\big)_{n\in\N}$ is uniformly distributed.
\end{theorem}

\begin{definition}
  A function $f:\R\to\R$ is called \emph{tempered} if for some $\ell\in\{0,1,2,\dots\}$ the first $\ell+1$
  derivatives of $f$ exist, are continuous on some interval $(a,\infty)$ and satisfy the following conditions:
  \begin{enumerate}
    \item $f^{(\ell+1)}(s)$ (eventually) decreases to $0$.
    \item $\displaystyle\lim_{s\to\infty} f^{(\ell)}(s)=\lim_{s\to\infty}sf^{(\ell+1)}(s)=\infty$.
  \end{enumerate}
  The sequence $\big(f(n)\big)_{n\in\N}$ is called a \emph{tempered sequence} and $\ell$ is called the \emph{degree} of $f$.
\end{definition}

Here are some examples of tempered sequences.
\begin{itemize}
  \item $x_n=\sum_{i=1}^kc_in^{\alpha_i}$, where $c_i,\alpha_i\in\R$, $c_k,\alpha_k>0$ and $\alpha_1<\cdots<\alpha_k$.
\item $x_n=cn^\alpha\big(\cos(\log^\beta n)+d\big)$, where $\alpha,\beta,c,d\in\R$ satisfy $\alpha,c>0$, $\beta<1$, and $d>1$.
\end{itemize}
Fej\'er's theorem together with van der Corput's \DT{} readily implies:
\begin{corollary}\label{cor_tempered}
  Let $f:\R\to\R$ be a tempered function. Then the sequence $\big(f(n)\bmod1\big)_{n\in\N}$ is uniformly distributed.
\end{corollary}
The notion of uniform distribution in the torus $\T$ can naturally be extended to a much more general setup.
For instance, one can consider sequences (functions) defined on more general domains and taking values in various metric spaces.
In order to make a generalization of Definition \ref{def_uniformdistribution}, one needs an appropriate replacement for the discrete interval $\{1,\dots,N\}$ in \eqref{eq_intro_uniformdistribution}.
Such replacement is provided by the notion of F\o lner sequence.
\begin{definition}
  Let $G$ be a locally compact Hausdorff group.
  A (left) \emph{F\o lner sequence} is a sequence $(F_N)_{N\in\N}$ of compact positive measure subsets of $G$ asymptotically invariant under left translations.
  More precisely:
  \begin{equation}\label{eq_folnerproperty}
  \lim_{N\to\infty}\frac{\lambda\big(gF_N\cap F_N\big)}{\lambda(F_N)}=1\qquad\qquad\forall g\in G
  \end{equation}
  where $gF_N:=\{gx:x\in F_N\}$ and $\lambda$ is the (left) Haar measure on $G$.
\end{definition}
Not every locally compact group has a F\o lner sequence.
Groups admitting  F\o lner sequences are called amenable.
It is well known that any abelian group $G$ is amenable.
We can now define uniform distribution of a function $u:G\to K$ from an amenable group $G$ to a compact group $K$.
\begin{definition}
  Let $G$ be a $\sigma$-compact locally compact amenable group with Haar measure $\lambda$, let $(F_N)_{N\in\N}$ be a F\o lner sequence in $G$, let $K$ be a compact group with normalized Haar measure $\mu$ and let $u:G\to K$ be a continuous function.
  We say that $u$ is \emph{uniformly distributed} (with respect to $(F_N)$) if, for every open set $U\subset K$ with boundary of measure $0$,
  \begin{equation}\label{eq_generalud}
    \lim_{N\to\infty}\frac{\lambda\big(\{t\in F_N:u(t)\in U\}\big)}{\lambda(F_N)}=\mu(U)
  \end{equation}
\end{definition}
In later chapters we will use the following general form of Theorem \ref{theorem_weyl}.

\begin{theorem}\label{theorem_weylgeneral}
 Let $G$ be a $\sigma$-compact locally compact amenable group with Haar measure $\lambda$, let $(F_N)_{N\in\N}$ be a F\o lner sequence in $G$, let $K$ be a compact abelian group with Haar measure $\mu$ and let $u:G\to K$ be a measurable function. The following are equivalent:
 \begin{enumerate}
   \item $u$ is uniformly distributed with respect to $(F_N)$.
   \item $$\lim_{N\to\infty}\frac1{\lambda(F_N)}\int_{F_N}f\big(u(t)\big)\dd\lambda(t)=\int_Kf(x)\dd \mu(x)\qquad\qquad\forall f\in C(K)$$
   \item $$\lim_{N\to\infty}\frac1{\lambda(F_N)}\int_{F_N}\chi\big(u(t)\big)\dd\lambda(t)=0\qquad\qquad\forall \chi\in\hat K\setminus\{0_{\hat K}\}$$
       where $\hat K$ is the Pontryagin dual of $K$.
 \end{enumerate}
\end{theorem}

\begin{proof}[Sketch of the proof \emph{(cf.\ {\cite[Exercise 1.1.1]{Tao12}})}]

  The implication (2)$\Rightarrow$(3) is trivial and the implication (3)$\Rightarrow$(2) follows directly from the Stone--Weierstrass theorem (and the fact that characters separate points).
  To prove the implication (1)$\Rightarrow$(2), let us assume that $u$ is uniformly distributed.
  Condition \eqref{eq_generalud} states that
  \begin{equation}\label{eq_proof_weylgeneral} \lim_{N\to\infty}\frac1{\lambda(F_N)}\int_{F_N}f\big(u(t)\big)\dd\lambda(t)=\int_Kf(x)\dd \mu(x)\end{equation}
  whenever $f$ is the indicator function of an open set $U$ with boundary of $0$ measure.
  Since \eqref{eq_proof_weylgeneral} is linear in $f$, it follows that \eqref{eq_proof_weylgeneral} holds whenever $f$ is a finite linear combination of indicator functions of open sets with boundary of $0$ measure.
  Now let $f\in C(K)$.
  For any $\epsilon>0$, one can find $f_1,f_2:K\to\R$ which are finite linear combinations of indicator functions of open sets with boundary of $0$ measure and satisfy $f_1\leq f\leq f_2$ and $\int_K f_2-f_1\dd\mu<\epsilon$.
  It follows that \eqref{eq_proof_weylgeneral} holds for continuous functions.

  To prove the converse implication (2)$\Rightarrow$(1), one can approximate the indicator function of any open set with boundary of $0$ measure from above and from below by continuous functions, using Urysohn's lemma.
  One can then proceed as in the proof of (1)$\Rightarrow$(2) above.
\end{proof}
Van der Corput's difference theorem can also be adapted to the following generality:
\begin{theorem}[General form of \DT]\label{thm_intro_vdCgeneral}
Let $G$ be a $\sigma$-compact locally compact amenable group with Haar measure $\lambda$, let $(F_N)_{N\in\N}$ be a F\o lner sequence in $G$, let $K$ be a compact abelian group with Haar measure $\mu$ and let $u:G\to K$ be a measurable function.
Assume that for every $d\in G\setminus\{1_G\}$, the function $t\mapsto u(td)-u(t)$ is uniformly distributed in $K$.
Then $u$ is uniformly distributed.
\end{theorem}

We will derive Theorem \ref{thm_intro_vdCgeneral} from the following result involving functions taking values in a Hilbert space which can be interpreted as yet another form of \DT{} (cf. the remark after Theorem 2.2 in \cite{Bergelson96}).

\begin{theorem}[\DT{} for Hilbert spaces]\label{theorem_vdctrickcontinuous}
Let $G$ be a $\sigma$-compact locally compact amenable group with Haar measure $\lambda$ and let $u:G\to H$ be a continuous bounded map into a Hilbert space $H$.
  Let $(F_N)_{N\in\N}$ be a F\o lner sequence in $G$.
  Assume that
  $$\lim_{D\to\infty}\frac1{\lambda(F_D)}\int_{F_D}\limsup_{N\to\infty}\left|\frac1{\lambda(F_N)}\int_{F_N}\big\langle
  u(sh),u(s)\big\rangle\dd\lambda(s)\right|\dd\lambda(h)=0$$
  Then we have
  $$\lim_{N\to\infty}\left\|\frac1{\lambda(F_N)}\int_{F_N}u(s)\dd\lambda(s)\right\|=0$$
\end{theorem}

While it is possible to adapt the proof of Theorem \ref{thm_vdCoriginal} to this general setting we instead choose to employ ideas, going back to Bass (\cite{Bass59}) that involve properties of positive definite functions.

\begin{proof}
Define a sequence $(a_N)_{N\in\N}$ in $H$ via the Bochner integral
$$a_N:=\frac1{\lambda(F_N)}\int_{F_N}u(s)\dd\lambda(s)$$
  Assume for the sake of a contradiction that the theorem is false.
  After passing, if needed, to a subsequence, we can assume that $a:=\lim_{N\to\infty}\|a_N\|$ exists and is not $0$.
  Define, for each $h\in G$
  \begin{equation}\label{eq_gamma}
    \gamma(h):=\lim_{N\to\infty}\frac1{\lambda(F_N)}\int_{F_N}\big\langle u(sh)-a_N,u(s)-a_N\big\rangle\dd\lambda(s)
  \end{equation}
  Passing to a further subsequence we can assume that the limit in \eqref{eq_gamma} exists for a countable dense subset of $h$.
  Since $\gamma$ is continuous, it follows that $\gamma$ can be defined for every $h\in G$.
  Using the F\o lner property \eqref{eq_folnerproperty} we can rewrite $\gamma(h)$ as
  \begin{equation}\label{eq_gamma2}
    \gamma(h)=-a+\lim_{N\to\infty}\frac1{\lambda(F_N)}\int_{F_N}\big\langle u(sh),u(s)\big\rangle\dd\lambda(s)
  \end{equation}
  Therefore
  \begin{equation}\label{eq_climgamma}
  \lim_{D\to\infty}\frac1{\lambda(F_D)}\int_{h\in F_D}\gamma(h)\dd\lambda(h)=-a
  \end{equation}
  On the other hand, it follows directly from \eqref{eq_gamma} that $\gamma$ is a positive definite function.
  Indeed, let $f:G\to\C$ be a function with finite support.
  Then
  $$\begin{aligned}
   \sum_{g,h\in G}f(g)\overline{f(h)}&\gamma(h^{-1}g)\\=& \sum_{g,h\in G}f(g)\overline{f(h)}\lim_{N\to\infty}\frac1{\lambda(F_N)}\int_{F_N}\big\langle u(nh^{-1}g)-a_N,u(n)-a_N\big\rangle\dd\lambda(n)\\=& \sum_{g,h\in G}f(g)\overline{f(h)}\lim_{N\to\infty}\frac1{\lambda(F_N)}\int_{F_N}\big\langle u(mg)-a_N,u(mh)-a_N\big\rangle\dd\lambda(m)\\=& \displaystyle\lim_{N\to\infty}\frac1{\lambda(F_N)}\int_{F_N}\left\|\sum_{g\in G}f(g)\big(u(mg)-a_N\big)\right\|^2\dd\mu(m)\\\geq&0
  \end{aligned}$$
  By the Naimark dilation theorem (see, for example, \cite[Theorem 4.8]{Paulsen02}), there exists a unitary representation $(U_g)_{g\in G}$ of $G$ in some Hilbert space $V$ and some vector $v\in V$ such that $\gamma(h)=\langle U_hv,v\rangle$.
  Invoking the mean ergodic theorem (see, for instance, \cite[Theorem 3.3]{Greenleaf73}) we conclude that
  \begin{eqnarray*}
    -a&=&\lim_{D\to\infty}\frac1{\lambda(F_D)}\int_{h\in F_D}\gamma(h)\dd\lambda(h)\\&=& \lim_{D\to\infty}\frac1{\lambda(F_D)}\int_{h\in F_D}\langle U_hv,v\rangle\dd\lambda(h)=\langle Pv,v\rangle=\|Pv\|^2
  \end{eqnarray*}
  where $P:V\to V$ is the orthogonal projection onto the subspace of invariant functions.
  The last equation contradicts the hypothesis that $a>0$.
\end{proof}

We can now derive Theorem \ref{thm_intro_vdCgeneral} from Theorem \ref{theorem_vdctrickcontinuous}:
\begin{proof}[Proof of Theorem \ref{thm_intro_vdCgeneral}]
We will use the Weyl criterion, statement (3) in Theorem \ref{theorem_weylgeneral}.
Let $\chi\in\hat K$ be nontrivial.
Let $v(s)=\chi\big(u(s)\big)\in\C$.
Viewing $\C$ as a one dimensional Hilbert space we have, for every $s,h\in G$, $h\neq1_G$
$$\big\langle v(sh),v(s)\big\rangle=\chi\big(u(sh)\big)\overline{\chi\big(u(s)\big)}=\chi\big(u(sh)-u(s)\big)$$
Since the map $s\mapsto u(sh)-u(s)$ is uniformly distributed, it follows from the Weyl criterion that
$$\lim_{N\to\infty}\frac1{\lambda(F_N)}\int_{F_N}\big\langle v(sh),v(s)\big\rangle\dd\lambda(s)= \lim_{N\to\infty}\frac1{\lambda(F_N)}\int_{F_N}\chi\big(u(sh)-u(s)\big)\dd\lambda(s)=0$$
for any $h\in G\setminus\{1_G\}$.
If we average the previous equation over all $h$ we will also trivially get $0$.
Therefore the conditions of Theorem \ref{theorem_vdctrickcontinuous} are met and hence
$$0=\lim_{N\to\infty}\left\|\frac1{\lambda(F_N)}\int_{F_N} v(s)\dd\lambda(s)\right\|= \lim_{N\to\infty}\left|\frac1{\lambda(F_N)}\int_{F_N}\chi\big(u(s)\big)\dd\lambda(s)\right|$$
It follows now from Theorem \ref{theorem_weylgeneral} that $u$ is uniformly distributed.
\end{proof}

\section{Well distribution}\label{sec_welldist}

A sequence $(u_n)_{n\in\N}$ taking values in the torus $\T$ is \emph{well distributed} if for every $0\leq a<b\leq 1$
\begin{equation}
\lim_{N\to\infty}\frac{\Big|\big\{n\in\{M+1,\dots,M+N\}:x_n\in[a,b)\big\}\Big|}N=b-a
\end{equation}
uniformly in $M$.
An equivalent definition is that $(u_n)_{n\in\N}$ is uniformly distributed along any F\o lner sequence in $\N$. 
It follows from the Weyl criterion that $(u_n)$ is well distributed if and only if for every $h\in\N$
$$\lim_{N\to\infty}\sup_{M\in\N}\left|\frac1N\sum_{n=M}^{M+N}e^{2\pi ihu_n}\right|=0$$
It is not hard to see that a trivial modification of the proof of Theorem \ref{thm_vdCoriginal} allows one to establish the following version of \DT{} for well distribution (one can also derive it from Theorem \ref{thm_intro_vdCgeneral}).

\begin{theorem}\label{thm_vdcwelldist}
Let $(x_n)_{n\in\N}$ be a sequence taking values in the torus $\T$.
Assume that for every $d\in\N$, the sequence $n\mapsto x_{n+d}-x_n$ is well distributed.
Then $(x_n)_{n\in\N}$ is well distributed.
\end{theorem}

\subsection{Well distribution and Hardy fields}
It follows from Theorem \ref{theorem_fejer} that the sequence $x_n=\sqrt{n}\bmod1$ is uniformly distributed.
Since the differences between consecutive terms $\sqrt{n+1}-\sqrt{n}$ converge to $0$, it is clear that $(x_n)_{n\in\N}$ is not well distributed.
On the other hand, it is easy to see that the sequence $n\alpha\bmod1$ is well distributed for any irrational $\alpha$.
What about $(n\alpha+\sqrt{n})\bmod1$?
As it turns out, this sequence is well distributed, and a quick way to prove this is to utilize (a special case of) Theorem \ref{theorem_vdctrickcontinuous}.
More generally, we have the following:
\begin{theorem}\label{theorem_welldistributionperturbation}
  Let $(v_n)_{n\in\N}$ be a non-decreasing sequence in $\R$ such that
  \begin{equation}\label{eq_theorem_welldistributionperturbation}
  \lim_{N\to\infty}\sup_{M\in\N}\frac{v_{M+N}-v_M}N=0\end{equation}
  Then for every $\alpha\notin\Q$ the sequence $u_n=n\alpha+v_n\bmod1$ is well distributed in $\T$.
\end{theorem}

\begin{remark}
\

\begin{enumerate}
  \item Theorem \ref{theorem_welldistributionperturbation} extends Theorem 1.3.3 in \cite{Kuipers_Niederreiter74}, where, under somewhat stronger assumptions, it is concluded that $u_n$ is uniformly distributed (by the way, Theorem 1.3.3 in \cite{Kuipers_Niederreiter74} also appears in van der Corput's seminal paper \cite{vdCorput31}).
  \item Under the slightly stronger condition that $v_{N+1}-v_N\to0$ as $N\to\infty$, it is actually true that $u_n=a_n+v_n$ is well distributed for any well distributed sequence $a_n$.
      This follows from Theorem 4.2.3 in \cite{Kuipers_Niederreiter74}.
\end{enumerate}
\end{remark}

Examples of sequences $(v_n)$ satisfying \eqref{eq_theorem_welldistributionperturbation} are $v_n=\sqrt{n}$, $v_n=\log n$ and $v_n=cn^a\big(\cos(\log^b n)+d\big)$, where $a,b,c,d\in\R$ satisfy $a,c>0$, $a,b<1$, $d>1$.
More generally, if $f:\R\to\R$ is a non-decreasing smooth function such that $f'(x)\to0$ as $x\to\infty$, then
$v_n=f(n)$ also satisfies \eqref{eq_theorem_welldistributionperturbation}.
Notice that $\log n$ is not even uniformly distributed in $\T$ but, in view of Theorem
\ref{theorem_welldistributionperturbation}, $n\alpha+\log n$ is well distributed.

We need the following elementary lemma:
\begin{lemma}\label{lemma_uniformcesaro}
  Let $(v_n)_{n\in\N}$ be a non-decreasing sequence in $\R$ satisfying \eqref{eq_theorem_welldistributionperturbation}.
  Then for every pair of sequences $(M_m),(N_m)$ of natural numbers such that $N_m-~M_m\to~\infty$ and every $d\neq0$ we have
  $$\lim_{m\to\infty}\frac1{N_m-M_m}\sum_{n=M_m+1}^{N_m}\exp\big(2\pi ik(v_{n+d}-v_d)\big)=1$$
\end{lemma}
\begin{proof}
  Observe that $\big|\exp(i\theta)-1\big|\leq|\theta|$ for any $\theta\in\R$ (geometrically this says that a
  chord is shorter than the corresponding arc).
  Then
  \begin{eqnarray*}\left|\sum_{n=M_m+1}^{N_m}\exp\big(2\pi
  ik(v_{n+d}-v_d)\big)-1\right|&\leq&\sum_{n=M_m+1}^{N_m}\left|\exp\big(2\pi ik(v_{n+d}-v_d)\big)-1\right|\\&\leq&
  2\pi k\sum_{n=M_m+1}^{N_m}\big|v_{n+d}-v_n\big|\\&=&2\pi k\sum_{i=1}^d\big(v_{N_m+i}-v_{M_m+i}\big)\end{eqnarray*}
  Dividing by $N_m-M_m$ and using \eqref{eq_theorem_welldistributionperturbation} $d$ times we conclude that
  $$\lim_{m\to\infty}\left|\frac1{N_m-M_m}\sum_{n=M_m+1}^{N_m}\exp\big(2\pi ik(v_{n+d}-v_d)\big)-1\right|=0$$
  as desired.
\end{proof}
\begin{proof}[Proof of Theorem \ref{theorem_welldistributionperturbation}]
We need to show that for all sequences $(M_m),(N_m)$ such that $N_m-M_m\to\infty$ and every
$k\neq0$ we have
\begin{equation}\label{eq_thm_welldistrib}
  \lim_{m\to\infty}\frac1{N_m-M_m}\sum_{n=M_m+1}^{N_m}\exp(2\pi iku_n)=0
\end{equation}
Let $w(n)=\exp(2\pi iku_n)$.
In order to prove \eqref{eq_thm_welldistrib}, we will use Theorem \ref{theorem_vdctrickcontinuous} with the Hilbert space being $\C$, the group $G=\Z$, the Haar measure $\lambda$ being simply the counting measure and the F\o lner sequence being defined by $F_m=\{M_m+1,\dots,N_m\}$.
We have
$$w(n+d)\overline{w(n)}=\exp\big(2\pi ik(u_{n+d}-u_n)\big)=\exp\big(2\pi ik(d\alpha+v_{n+d}-v_n)\big)$$
Taking Ces\`aro limits and appealing to Lemma \ref{lemma_uniformcesaro} we obtain
$$\begin{aligned}
   \lim_{m\to\infty}\frac1{|F_m|}\sum_{n\in F_m}w(n+d)&\overline{w(n)}\\
  =&\exp(2\pi
  ikd\alpha)\lim_{m\to\infty}\frac1{|F_m|}\sum_{n\in F_m}\exp\big(2\pi ik(v_{n+d}-v_n)\big) \\
  =&\exp(2\pi ikd\alpha)
\end{aligned}$$
Taking the Ces\`aro limit in $d$ we deduce that
$$\lim_{D\to\infty}\frac1D\sum_{d=1}^D\lim_{m\to\infty}\frac1{N_m-M_m}\sum_{n=M_m+1}^{N_m}w(n+d)\overline{w(n)}=0$$
Now \eqref{eq_thm_welldistrib} follows from Theorem \ref{theorem_vdctrickcontinuous} and this finishes the proof.
\end{proof}

Observe that if $(au_n+bv_n)_{n\in\N}$ is well distributed in $\T$ for every $(a,b)\in\Z^2\setminus\{(0,0)\}$, then it follows from Weyl's criterion (Theorem \ref{theorem_weylgeneral}) that the sequence $(u_n,v_n)_{n\in\N}$ is well distributed in $\T^2$.
This observation should be juxtaposed with the following corollary of Theorem \ref{theorem_welldistributionperturbation}.
\begin{corollary}
  There exists a sequence $(u_n,v_n)\in\T^2$ which is not uniformly distributed but such that for any $a,b\in\Z$ with
  $a\neq0$ we have that $au_n+bv_n$ is well distributed in $\T$.
\end{corollary}
\begin{proof}
  Let $u_n=n\alpha$ for some irrational $\alpha\in\R$, let $v_n=\log n$ and apply Theorem \ref{theorem_welldistributionperturbation}.
\end{proof}

Recall that $\Delta^sf$ denotes the $s$-th iterated discrete derivative of $f$.
Bootstrapping Theorem \ref{theorem_welldistributionperturbation} with the help of \DT{} we obtain the following more general result.
\begin{corollary}\label{cor_theorem_welldistributionperturbation}
  Let $s\in\N$ and let $p\in\R[x]$ be a polynomial with $\deg p=s$ whose leading coefficient is irrational.
  Let $f:\N\to\R$ be such that
  \begin{equation}\label{eq_cor_theorem_welldistributionperturbation}
    \lim_{N\to\infty}\sup_{M\in\N}\left|\frac1N\sum_{n=M}^{M+N}\Delta^sf(n)\right|=0
  \end{equation}
  Then the sequence
  $\big(p(n)+f(n)\bmod1\big)_{n\in\N}$ is well distributed.
\end{corollary}
\begin{proof}
  We proceed by induction on $s$, the case $s=1$ being Theorem \ref{theorem_welldistributionperturbation}.
  Next assume that $s>1$ and that the result holds for $s-1$.
  Let $u_n=p(n)+f(n)\bmod1$ and let $d\in\N$ be arbitrary.
  The difference $u_{n+d}-u_n$ can be written as $\tilde p(n)+\tilde f(n)$, where $\tilde p(n)=p(n+d)-p(n)$ is a polynomial of
  degree $s-1$ with leading coefficient irrational and $\tilde f(n)=\Delta f(n)+\Delta f(n+1)+\cdots+\Delta f(n+d-1)$ is a sequence satisfying \eqref{eq_cor_theorem_welldistributionperturbation} with $s-1$ instead of $s$.
  By the induction hypothesis, $u_{n+d}-u_n$ is well distributed for every $d\in\N$.
  It follows from \DT{} (in the form of Theorem \ref{thm_vdcwelldist}) that $u_n$ is well distributed as well.
\end{proof}

Observe that if $f:\R\to\R$ is a smooth function such that $f^{(s)}(x)\to0$ as $x\to\infty$, then the sequence $f(n)$ satisfies \eqref{eq_cor_theorem_welldistributionperturbation}.

We will now demonstrate the usefulness of Corollary \ref{cor_theorem_welldistributionperturbation} by providing a relatively short proof of a result of Boshernitzan pertaining to Hardy fields.
First we need to introduce some relevant definitions.
Two real valued functions $f,g$ defined for all large enough $x\in\R$ are \emph{equivalent} if there exists $M\in\R$ such that $f(x)=g(x)$ for all $x>M$.
A \emph{germ} is an equivalence class of functions under this equivalence relation.
\begin{definition}
  A Hardy field is a collection of germs of smooth functions $f:\R\to\R$ that forms a field (with respect to the operations of pointwise addition and multiplication) and is closed under differentiation.
\end{definition}
Observe that any function representing a non-zero germ in a Hardy field must be eventually non-zero (since it is invertible).
It follows that any non-zero element of a Hardy field is eventually either positive or negative.
 Since for any Hardy field function $f$, the derivative $f'$ belongs to the same Hardy field, we deduce that $f$ is eventually monotone, and hence the limit $\lim_{x\to\infty}f(x)$ always exists in $[-\infty,\infty]$.
 Many familiar functions belong to some Hardy field, including any function obtained from $e^x$, $\log x$ and polynomial functions by composition and standard arithmetic operations.
 On the other hand, (non-constant) periodic functions do not belong to any Hardy field.

 A function $f:\R\to\R$ is called \emph{subpolynomial} if there exists $n\in\N$ such that $f(x)/x^n\to0$ as $x\to\infty$.
\begin{theorem}[Boshernitzan, \cite{Boshernitzan94}]\label{theorem_Boshernitzan}
  Assume $f:\R\to\R$ belongs to a Hardy field and is subpolynomial. Then the sequence $\big(f(n)\bmod1\big)_{n\in\N}$ is uniformly distributed in $\T$ if and only if for every $p(x)\in\Q[x]$ we have
  \begin{equation}\label{eq_boshernitzan}
\lim_{x\to\infty}\frac{\log x}{|f(x)-p(x)|}=0
  \end{equation}
\end{theorem}
\begin{proof}
First assume that
$$\lim_{x\to\infty}\frac{|f(x)-p(x)|}{\log x}=a<\infty$$
for some $p\in\Q[x]$.
Then $f(x)=\big(a+g(x)\big)\log x+p(x)$ for some function $g(x)$ which approaches $0$ as $x\to\infty$.
Let $M$ be a common multiple of the denominators of the coefficients of $p$.
Then $Mf(n)=\big(Ma+g(n)\big)\log n\bmod1$ for every $n\in\N$ grows too slowly to be uniformly distributed (see, for example, Theorem 1.2.6 in \cite{Kuipers_Niederreiter74}).

We move now to the proof of the converse direction.
 Let $f$ be a function satisfying the conditions of the theorem. 
  Let $d=d(f)$ be the smallest integer such that $f(x)/x^d$ is bounded and let $L=\lim f(x)/x^d$. (The limit exists because $f$ belongs to a Hardy field.)

  We will prove the result by induction on $d$.
  If $L$ is irrational, then the uniform distribution (and indeed well distribution) of $\big(f(n)\bmod1\big)_{n\in\N}$ follows directly from either Theorem \ref{theorem_welldistributionperturbation} or Corollary \ref{cor_theorem_welldistributionperturbation}.

  Observe that for any $p\in\Q[x]$, the sequence $\big(p(n)\bmod1\big)_{n\in\N}$ is periodic.
  Hence $\big(f(n)+p(n)\bmod1\big)_{n\in\N}$ is uniformly distributed if and only if $\big(f(n)\bmod1\big)_{n\in\N}$ is.
  In particular, if $L\in\Q\setminus\{0\}$, then the uniform distribution of $\big(f(n)\bmod1\big)_{n\in\N}$ is equivalent to that of $\big(g(n)\bmod1\big)_{n\in\N}$, with $g(x)=f(x)-Lx^d$ satisfying the conditions of the theorem and either $d(g)<d(f)$ or $L(g)=0$.
  Therefore it suffices to assume that $L=0$.

  Assume first that $d=1$.
  Then L'H\^opital's rule implies that $f'(x)\to0$ and, taking into account that $f(x)/\log(x)\to\infty$, we obtain that $xf'(x)\to\infty$ as $x\to\infty$.
  It now follows from Theorem \ref{theorem_fejer} that $\big(f(n)\bmod1\big)_{n\in\N}$ is uniformly distributed in $\T$.

  Next we deal with the case $d=2$ (a typical example is $f(x)=x\log x$).
  The reason why we can not use induction here is that the derivative $f'(x)$ may not be uniformly distributed $\bmod\  1 $.
  By L'H\^opital's rule we have that $f'(x)\to\infty$ and $f''(x)\to0$.
  Therefore also $f'(x)/x^2f''(x)\to0$ and now the uniform distribution of $\big(f(n)\bmod1\big)_{n\in\N}$ follows from Exercise 1.2.26 in \cite{Kuipers_Niederreiter74}\footnote{Curiously enough, this exercise is based on yet another result of van der Corput, Theorem 1.2.7 in \cite{Kuipers_Niederreiter74}.}.

  Finally, to deal with $d>2$, we apply Theorem \ref{thm_vdCoriginal}.
  Take an arbitrary $h\in\N$ and let $f_h(x)=f(x+h)-f(x)$.
  Since $f$ is a subpolynomial function in a Hardy field, one can show that $f_h$ is also a subpolynomial function in a Hardy field and that $f_h(x)/x^{d-1}\to0$.
  Therefore $d(f_h)\leq d-1$.
  On the other hand, $f_h(x)/x\to\infty$, and in fact $|f_h(x)-p(x)|/x\to\infty$ for any $p\in\Q[x]$.
  This implies that $\log(x)/|f_h(x)-p(x)|\to0$ and now the result follows by induction.
\end{proof}

Recently, Theorem \ref{theorem_Boshernitzan} was extended to sequences of the form $f(p_n)$ where $f$ is a subpolynomial Hardy field function and $p_n$ is the $n$-th prime:

\begin{theorem}[see {\cite[Theorem 3.1]{Bergelson_Son15}}]
  Assume that $f:\R\to\R$ belongs to a Hardy field and is subpolynomial and let $(p_n)_{n\in\N}$ be the increasing sequence listing all the prime numbers. Then the sequence $\big(f(p_n)\bmod1\big)_{n\in\N}$ is uniformly distributed in $\T$ if and only if for every $p(x)\in\Q[x]$ we have
  $$\lim_{x\to\infty}\frac{\log x}{|f(x)-p(x)|}=0$$
\end{theorem}
See also Corollary \ref{cor_besibosher} below for another result related to Theorem \ref{theorem_Boshernitzan}.
\subsection{Cigler's characterization of well distribution and some of its corollaries}
For completeness of the picture, we present in this subsection a result due to Cigler \cite{Cigler66} that provides an explanation for the dynamical underpinnings of the notion of well distribution.
We also present another result of Cigler \cite{Cigler68} which states that for any tempered function $f$, the sequence $f(n)$ can not be well distributed, complementing Corollary \ref{cor_tempered}.

\begin{theorem}\label{thm_welldistrdynamical}
Let $x=(x_n)_{n\in\N}$ be a sequence in $\T$.
Consider $x$ as an element of $\T^\N$ and let $S:\T^\N\to\T^\N$ be the shift map.
  Then $(x_n)_{n\in\N}$ is well distributed in $\T$ if and only if for any $S$-invariant measure $\mu$ on the orbit closure $\overline{\{S^nx:n\in\N\}}$ of $x$, the pushforward $\pi_*\mu$ by the projection $\pi:\T^\N\to\T$ onto the first coordinate is the Lebesgue measure $\lambda$ on $\T$.
\end{theorem}
\begin{proof}
  Assume first that $(x_n)_{n\in\N}$ is well distributed and let $\mu$ be an invariant measure on $X:=\overline{\{S^nx:n\in\N\}}$.
  Let $f\in C(\T)$ be arbitrary.
  We need to show that $\int_Xf\circ\pi\dd\mu=\int_\T f\dd\lambda$.
  For each $y\in X$, let $k_i(y)$ be a sequence in $\N$ such that $S^{k_i(y)}x\to y$ as $i\to\infty$.
  We have
  \begin{eqnarray*}
    \int_Xf\circ\pi\dd\mu&=&\int_Xf\big(\pi(y)\big)\dd\mu(y)\\&=& \lim_{N\to\infty}\int_X\frac1N\sum_{n=1}^Nf\big(\pi(S^ny)\big)\dd\mu(y)\\&=& \lim_{N\to\infty}\int_X\lim_{i\to\infty}\frac1N\sum_{n=1}^Nf\big(\pi(S^{n+k_i(y)}x)\big)\dd\mu(y)\\&=& \lim_{N\to\infty}\int_X\lim_{i\to\infty}\frac1N\sum_{n=1+k_i(y)}^{N+k_i(y)}f\big(\pi(S^nx)\big)\dd\mu(y)\\&=& \lim_{N\to\infty}\int_X\lim_{i\to\infty}\frac1N\sum_{n=1+k_i(y)}^{N+k_i(y)}f(x_n)\dd\mu(y)
  \end{eqnarray*}
  Since $(x_n)_{n\in\N}$ is well distributed, for any $\epsilon>0$, if $N$ is large enough we have
  $$\left|\frac1N\sum_{n=1+k_i(y)}^{N+k_i(y)}f\big(\pi(S^nx)\big)-\int_\T f(x)\dd\lambda(x)\right|<\epsilon$$
  regardless of $i\in\N$ and $y\in X$.
  Therefore, we deduce that $\int_Xf\circ\pi\dd\mu=\int_\T f(x)\dd\lambda(x)$.

  Now assume that for every invariant measure $\mu$ on $X$ we have $\pi_*\mu=\lambda$.
  Let $f\in C(X)$ and assume, for the sake of a contradiction, that
  $$A_N:=\frac1N\sum_{n=a_N+1}^{a_N+N}f(x_n)$$
  does not converge to $\int_\T f\dd\lambda$ for some sequence $(a_N)_{N\in\N}$.
  Passing, if necessary, to a subsequence we can assume that $\lim_{k\to\infty}A_{N_k}=A$ exists and does not equal $\int_\T f\dd\lambda$.
  Let
  $$\mu_k=\frac1{N_k}\sum_{n=a_{N_k}+1}^{a_{N_k}+N_k}\delta_{S^nx}$$
  be the corresponding average of point masses at the orbit of $x$.
  Passing to a further subsequence we can assume that $\mu_k$ converges as $k\to\infty$ to some $S$-invariant measure $\mu$ in the weak$^*$ topology.
  By the assumption, $\pi_*\mu=\lambda$ and therefore
  \begin{eqnarray*}
    \int_\T f\dd\lambda&=&\int_X f\big(\pi(y)\big)\dd\mu(y)\\&=& \lim_{k\to\infty}\frac1{N_k}\sum_{n=a_{N_k}+1}^{a_{N_k}+N_k}f\big(\pi(S^nx)\big)\\&=& \lim_{k\to\infty}\frac1{N_k}\sum_{n=a_{N_k}+1}^{a_{N_k}+N_k}f(x_n)=A
  \end{eqnarray*}
  This contradiction finishes the proof.
\end{proof}
\begin{corollary}
  For every tempered function $f:\R\to\R$, the sequence $(f(n)\bmod1)_{n\in\N}$ in $\T$ is not well distributed.
\end{corollary}
\begin{proof}
Let $x_n=f(n)\bmod1$ and $x=(x_n)_{n\in\N}$ be the corresponding sequence in $\T^\N$.
Let $\ell$ be the degree of $f$, so that $f^{(\ell)}$ goes to infinity but $f^{(\ell+1)}$ goes to $0$.
Then each of the functions $f,f',\dots,f^{(\ell)}$ is tempered and, moreover, each non-zero linear combination of these functions (with integer coefficients) is tempered.
In view Corollary \ref{cor_tempered} and Theorem \ref{theorem_weylgeneral}, the tuple $\big(f(n),f'(n),\dots,f^{(\ell)}(n)\big)$ is uniformly distributed in $\T^{\ell+1}$.
In particular, that sequence is dense in $\T^{\ell+1}$ so there exists some sequence $n_k\to\infty$ such that $f^{(i)}(n_k)\to0$ as $k\to\infty$ for all $i=0,\dots,\ell$.

For each $h\in\N$ we now get
$$f(n_k+h)=\sum_{i=0}^\ell\frac{f^{(i)}(n_k)}{i!}h^i+\frac{f^{(\ell+1)}(t)}{(\ell+1)!}h^{(\ell+1)}$$
for some $t\in[n_k,n_k+h]$.
After reducing modulo $1$ we deduce that $x_{n_k+h}\to0$ as $k\to\infty$ for every $h\in\N$.
Translating to a dynamical viewpoint, this means that the point $(0,0,\dots)$ belongs to the orbit closure of $x$.
Since this point is fixed under the shift map $S$, it supports a point mass which is an invariant measure.
The projection under $\pi:\T^\N\to\T$ of this measure is then the point mass at $0$ which is not the Lebesgue measure.
It now follows from Theorem \ref{thm_welldistrdynamical} that $x$ is not well distributed.
\end{proof}
\begin{corollary}
  Almost every sequence in $\T^\N$ (with the product Lebesgue measure) is not well distributed.
\end{corollary}
\begin{proof}
Notice that the shift $S$ on $\T^\N$ is ergodic with respect to the product Lebesgue measure, and hence almost every $\omega\in\T^\N$ has a dense orbit (under $S$) in $\T^\N$.
  In particular, the orbit closure of $\omega$ contains the fixed point $(0,0,\dots)$ which supports an invariant (atomic) measure $\delta$ whose projection under $\pi$ is not the Lebesgue measure.
\end{proof}

\section{Uniform distribution along Besicovitch almost periodic subsequences}\label{sec_besicovitch}

The results obtained in this section heavily depend on ideas introduced in papers by Mend\`es France \cite{MendesFrance73} and Daboussi and Mend\`es France \cite{Daboussi_MendesFrance74}.
In particular, we establish a refined version of \DT{} involving Besicovitch almost periodic functions and derive some interesting new applications.
\begin{definition}\label{def_besicovitch}\
\begin{itemize}
    \item For a subset $A\subset\N$, we define its upper density and Banach lower density, respectively, by
    $$\bar d(A)=\limsup_{N\to\infty}\frac{|A\cap\{1,\dots,N\}|}N$$ $$d_*(A)=\lim_{N\to\infty}\inf_{M\in\N}\frac{|A\cap\{M+1,\dots,M+N\}|}N$$
  \item A function $f:\N\to\C$ is \emph{Besicovitch almost periodic} if for every $\epsilon>0$ there exist $k\in\N$
      and $\alpha_1,\dots,\alpha_k\in\T$, $c_1,\dots,c_k\in\C$ such that
  \begin{equation}\label{eq_besicovitchalmostperiodic}\limsup_{N\to\infty}\frac1N\sum_{n=1}^N\left|f(n)-\sum_{j=1}^kc_je^{2\pi i\alpha_jn}\right|<\epsilon\end{equation}
    \item A function $f:\N\to\C$ is \emph{uniform Besicovitch almost periodic} if for every $\epsilon>0$ there exist $k\in\N$
      and $\alpha_1,\dots,\alpha_k\in\T$, $c_1,\dots,c_k\in\C$ such that
  $$\limsup_{N\to\infty}\sup_{M\in\Z}\frac1N\sum_{n=M+1}^{M+N}\left|f(n)-\sum_{j=1}^kc_je^{2\pi i\alpha_jn}\right|<\epsilon$$
  \item An increasing sequence $(n_k)_{k\in\N}$ in $\N$ is a \emph{(uniform) Besicovitch sequence} if the indicator function
      $1_A$ of the set $A=\{n_k:k\in\N\}$ is (uniform) Besicovitch almost periodic and $A$ has positive upper density (positive lower Banach density).
\end{itemize}
\end{definition}
The following lemma provides a way to check that certain functions are Besicovitch almost periodic.
\begin{lemma}[{\cite{Besicovitch55}}, cf.\ also {\cite[Corollary 3.21]{Bellow_Losert85}}]\label{lemma_besi}
  A function $f:\N\to\C$ is (uniform) Besicovitch almost periodic if there exists a compact abelian group $K$, a
  point $a\in K$ and a Riemann integrable function $\phi:K\to\C$ such that $f(n)=\phi(na)$.
\end{lemma}

In the next proposition we collect some examples of Besicovitch sequences.
\begin{proposition}\label{prop_apexample}\leavevmode
  \begin{enumerate}
    \item For any $\alpha>0$, the sequence $k\mapsto\lfloor k\alpha\rfloor$ is a uniform Besicovitch sequence.
    \item The increasing sequence $(n_k)_{k\in\N}$ of squarefree numbers is a Besicovitch sequence.
    \item More generally, for any set $Q\subset\N$ such that $\sum_{m\in Q}1/m<\infty$, the increasing sequence  $(n_k)_{k\in\N}$ listing the numbers $n$ without a divisor in $Q$ is a Besicovitch sequence.
    \item For any irrational $\alpha\in\R$, the set $A:=\big\{n\in\N:\gcd(n,\lfloor n\alpha\rfloor)=1\big\}$ has a Besicovitch almost periodic indicator function, therefore the increasing enumeration of the elements of $A$ is a Besicovitch sequence.
  \end{enumerate}
\end{proposition}
\begin{proof}\

  \begin{enumerate}
    \item Let $\beta=1/\alpha$ and let $A=\{\lfloor k\alpha\rfloor:k\in\N\}$.
    Observe that $\lfloor k\alpha\rfloor=n\iff n\beta\leq k<n\beta+\beta$ and therefore $n\in A\iff
    n\beta\bmod1\in(1-\beta,1]$.
    Let $\phi:\T\to\C$ be the indicator function of the set $(1-\beta,1]$.
    I   t follows from Lemma
    \ref{lemma_besi} that the sequence $k\mapsto\lfloor k\alpha\rfloor$ is uniform Besicovitch.
    \stepcounter{enumi}
    \item For each $M\in\N$, let $Q_M$ be the set of natural numbers not divisible by an element of $Q\cap[1,M]$.
    The indicator function $1_{Q_M}$ is periodic and hence can be written as a finite sum
    $\sum_{j=1}^kc_je^{2\pi i\alpha_jn}$ (with all the $\alpha_j\in\Q$).

    Let 
    $\epsilon>0$ and choose $M$ so that $\sum_{m>M}1_Q(m)/m<\epsilon$.
    It is clear that $Q\subset Q_M$.
    Moreover any number in $Q_M\setminus Q$ must be divisible by some element of $Q\cap[M+1,\infty)$.
    Therefore, for each $N\in\N$, the cardinality of $\{1,\dots,N\}\cap(Q_M\setminus Q)$ is at most
    $\sum_{m=M+1}^\infty1_Q(m) N/m<N\epsilon$.
    We conclude that
    $$\limsup_{N\to\infty}\frac1N\sum_{n=1}^N\left|1_Q(n)-\sum_{j=1}^kc_je^{2\pi i\alpha_jn}\right|<\epsilon$$
   \item This is the content of the paper \cite{Spilker00}.\qedhere
  \end{enumerate}
\end{proof}

We have the following result (cf.\ \cite{MendesFrance73,Daboussi_MendesFrance74,Rauzy76}) which can be viewed as a refinement of Theorem \ref{thm_vdCoriginal}.

\begin{theorem}\label{theorem_besicovitchsubsequence}
  Let $(u_n)_{n\in\N}$ be a sequence in $\T$.
  If for every $h\in\N$ the difference $n\mapsto u_{n+h}-u_n$ is uniformly distributed in $\T$ then for any Besicovitch sequence $(n_k)_{k\in\N}$ the sequence $(u_{n_k})_{k\in\N}$ is uniformly distributed.
  Moreover, if for every $h\in\N$ the sequence $(u_{n+h}-u_n)_{n\in\N}$ is well distributed and $(n_k)_{k\in\N}$ is uniform Besicovitch, then $(u_{n_k})_{k\in\N}$ is well distributed.
\end{theorem}

We will actually prove a multidimensional version of Theorem \ref{theorem_besicovitchsubsequence}.
First, we will need the following multidimensional version of Definition \ref{def_besicovitch}.
\begin{definition}Let $d\in\N$.
\begin{itemize}
\item For a set $A\subset\N^d$ we define the \emph{upper density} by the formula
$$\bar d(A)=\limsup_{N\to\infty}\frac{\big|A\cap\{1,\dots,N\}^d\big|}{N^d}$$
and the \emph{lower Banach density} by the formula
$$d_*(A)=\liminf_{N\to\infty}\inf_{(M_1,\dots,M_d)\in\N^d}\frac{\left|A\cap\prod_{i=1}^d\{M_i+1,\dots,M_i+N\}\right|}{N^d}$$

  \item A function $f:\N^d\to\C$ is \emph{Besicovitch almost periodic} if for any $\epsilon>0$ there exists $\ell\in\N$ and $\alpha_1,\dots,\alpha_\ell\in\T^d$, $c_1,\dots,c_\ell\in\C$ such that
  $$\limsup_{N\to\infty}\frac1{N^d}\sum_{n\in[1,N]^d}\left|f(n)-\sum_{j=1}^\ell c_j\exp2\pi i\langle\alpha_j,n\rangle\right|<\epsilon$$
  \item We say that $f:\N^d\to\C$ is \emph{uniform Besicovitch almost periodic} if one can replace the above limit by its uniform version: $$\limsup_{N\to\infty}\sup_{m\in\N^d}\frac1{N^d}\sum_{n\in m+[1,N]^d}\left|f(n)-\sum_{j=1}^\ell c_j\exp2\pi i\langle\alpha_j,n\rangle\right|<\epsilon$$
      where $[1,N]=\{1,\dots,N\}$.
  \item A sequence $n:\N^d\to\N^d$ is \emph{increasing} if whenever $a,b\in\N^d$ and for some coordinate $j\in\{1,\dots,d\}$, one has $a_j\leq b_j$, then also $n(a)_j\leq n(b)_j$.
  \item An increasing sequence $n:\N^d\to\N^d$ is a \emph{(uniform) Besicovitch sequence} if the indicator function $1_A$ of the set $A:=\{n(k):k\in\N^d\}$ is (uniform) Besicovitch almost periodic and $A$ has positive upper density (positive lower Banach density).
\end{itemize}
\end{definition}

\begin{theorem}\label{thm_besimultidim}
  Let $d\in\N$, let $K$ be a compact abelian group and let $u:\N^d\to K$ be such that for any $h\in\N^d$ the map $n\mapsto u(n+h)-u(n)$ is uniformly distributed in $K$.
  Let $n:\N^d\to\N^d$ be an increasing Besicovitch sequence.
  Then the sequence $k\mapsto u\big(n(k)\big)$ is uniformly distributed in $K$.

  Moreover, if the maps $n\mapsto u(n+h)-u(n)$ are well distributed and $n$ is uniform Besicovitch, then $k\mapsto u\big(n(k)\big)$ is well distributed in $K$.
\end{theorem}
\begin{proof}

  Let $\chi\in\hat K\setminus\{0_{\hat K}\}$ be an arbitrary non-constant character on the Pontryagin dual $\hat K$ of the compact abelian group $K$.
  We need to show that
  $$\lim_{N\to\infty}\frac1{N^d}\sum_{k\in m+[1,N]^d}\chi\Big(u\big(n(k)\big)\Big)=0$$
  uniformly in $m\in\N^d$.
  Let $A=\{n(k):k\in\N^d\}$ and let $1_A$ be the indicator function of $A$.
  Also, for each $m\in\N^d$ and $N\in\N$, let $I(m,N)=m+[1,N]^d=\prod_i[m_i,m_i+N]$ and let $\tilde I(m,N)=\prod_i[n(m_i),n(m_i+N)]$.
  Observe that $n^{-1}\big(\tilde I(m,N)\big)=I(m,N)$ because $n$ is increasing, and therefore $n\big(I(m,N)\big)=A\cap\tilde I(m,N)$.
  Hence
  $$\liminf_{N\to\infty}\frac{N^d}{|\tilde I(m,N)|}=\liminf_{N\to\infty}\frac{|A\cap\tilde I(m,N)|}{|\tilde I(m,N)|}\geq d_*(A)>0$$
  Moreover
  $$\frac1{N^d}\sum_{k\in I(m,N)}\chi\Big(u\big(n(k)\big)\Big)= \frac1{N^d}\sum_{n\in\tilde I(m,N)}1_A(n)\chi\big(u(n)\big)$$
  Fix $\epsilon>0$ and let $\alpha_1,\dots,\alpha_s\in\T^d$ and $c_1,\dots,c_s\in\C$ such that
  $$\limsup_{N\to\infty}\frac1{|\tilde I(m,N)|}\sum_{n\in\tilde I(m,N)}\left|1_A(n)-\sum_{j=1}^sc_j\exp(2\pi i\langle\alpha_j,n\rangle)\right|<\epsilon$$
  We deduce that
  \begin{eqnarray*}
    &&\limsup_{N\to\infty}\frac1{N^d}\left|\sum_{n\in\tilde I(m,N)}1_A(n)\chi\big(u(n)\big)\right|\\&=& \limsup_{N\to\infty}\frac1{N^d}\left|\sum_{n\in\tilde I(m,N)}\left(1_A(n)-\sum_{j=1}^sc_j\exp(2\pi
    i\langle\alpha_j,n\rangle)+\sum_{j=1}^sc_j\exp(2\pi i\langle\alpha_j,n\rangle)\right)\chi\big(u(n)\big)\right|\\&<&
    \epsilon\limsup_{N\to\infty}\frac{|\tilde I(m,N)|}{N^d}+\limsup_{N\to\infty}\frac1{N^d}\left|\sum_{n\in\tilde I(m,N)}\sum_{j=1}^sc_j\exp\big(2\pi i\langle\alpha_j,n\rangle\big)\cdot\chi\big(u(n)\big)\right|\\&\leq& \frac\epsilon{d_*(A)}+\sum_{j=1}^sc_j\frac1{d_*(A)}\limsup_{N\to\infty}\left|\frac1{\big|\tilde I(m,N)\big|}\sum_{n\in\tilde I(m,N)}\exp\big(2\pi i\langle\alpha_j,n\rangle\big)\cdot\chi\big(u(n)\big)\right|
  \end{eqnarray*}
  It suffices now to show that the uniform Ces\`aro limit of the sequence $v:n\mapsto\exp\big(2\pi i\langle\alpha_j,n\rangle\big)\cdot\chi\big(u(n)\big)$ is $0$.
  To show this we will use Theorem \ref{thm_vdcwelldist}.
  We have
  $$v(n+h)\overline{v(n)}=\exp\big(2\pi i\langle\alpha_j,h\rangle\big)\cdot\chi\big(u(n+h)-u(n)\big)$$
  Therefore, averaging over $n$, with $h$ fixed we obtain
  \begin{multline}\label{eq_besicovitchmultidim}
    \frac1{\big|\tilde I(m,N)\big|}\sum_{n\in\tilde I(m,N)}v(n+h)\overline{v(n)}\\=\exp\big(2\pi i\langle\alpha_j,h\rangle\big)\frac1{\big|\tilde I(m,N)\big|}\sum_{n\in\tilde I(m,N)}\chi\big(u(n+h)-u(n)\big)
  \end{multline}
  It follows from Weyl's criterion, together with the hypothesis that $u(n+h)-u(n)$ is well distributed, that the limit as $N\to\infty$ of the quantity in \eqref{eq_besicovitchmultidim} is $0$.
It now follows from \DT{} (in the form of Theorem \ref{thm_vdcwelldist}) that the uniform Ces\`aro limit of $v(n)$ is $0$ and we are done.
\end{proof}
The argument used in part (1) of Proposition \ref{prop_apexample} can be utilized to show the following:
\begin{proposition}\label{prop_besicovitch}
  Let $\alpha,\beta\in[1,\infty)$.
  Then the sequence $n:(k_1,k_2)\mapsto(\lfloor k_1\alpha\rfloor,\lfloor k_2\beta\rfloor)$ is uniform Besicovitch.
\end{proposition}
Putting Proposition \ref{prop_besicovitch} together with Theorem \ref{thm_besimultidim}, we conclude that, for instance, the sequence $(n,m)\mapsto(\alpha\lfloor n\beta\rfloor^3,\sqrt{\lfloor m\gamma\rfloor})$ is uniformly distributed in $\T^2$ for any irrational $\alpha\in\R$ and $\beta,\gamma\geq1$.

Observe that Theorem \ref{theorem_besicovitchsubsequence} does not remain true if one only requires that $u_n$ is itself uniformly distributed.
For instance, for any irrational $\alpha\in(0,1)$, the sequence $n\mapsto n\alpha\bmod1$ is uniformly distributed, and, in view of Proposition \ref{prop_apexample}, the increasing sequence $k\mapsto\lfloor k\alpha\rfloor$ is a Besicovitch sequence.
However, it is not hard to show that $\lfloor k/\alpha\rfloor\alpha\bmod1$ is not uniformly distributed (indeed, this sequence always lands in the interval $(1-\alpha,1)$).
On the other hand, it can be proved that $\lfloor k\alpha\rfloor\beta\bmod1$ \emph{is} uniformly distributed whenever $\alpha\beta$ and $\beta$ are irrational (cf.\ \cite[Theorem 5.1.8]{Kuipers_Niederreiter74}).

By looking more closely at the proof of Theorem \ref{theorem_besicovitchsubsequence}, one notices that the condition that $u_{n+h}-u_n$ is uniformly distributed for every $h\in\N$ can be replaced with the (weaker) condition that $u_n+n\alpha$ is uniformly distributed for every $\alpha\in\T$.
More precisely, we only need $u_n+n\alpha$ to be uniformly distributed for the frequencies $\alpha=\alpha_j$ that appear in the periodic approximation \eqref{eq_besicovitchalmostperiodic} of the (indicator function of the) Besicovitch sequence $(n_k)$.
This improvement of Theorem \ref{theorem_besicovitchsubsequence} was obtained by Daboussi and Mend\`es France in \cite{Daboussi_MendesFrance74} (the authors called the set of such frequencies the \emph{spectrum} of $(u_n)$).

Of particular interest is the case when all the frequencies needed to approximate the (indicator function of the) Besicovitch sequence $(n_k)$ are rational:

\begin{definition}
  \
  \begin{itemize}
  \item A function $f:\N\to\C$ is \emph{rational Besicovitch almost periodic} if for every $\epsilon>0$ there exist $k\in\N$
      and $\alpha_1,\dots,\alpha_k\in\Q$, $c_1,\dots,c_k\in\C$ such that
  $$\limsup_{N\to\infty}\frac1N\sum_{n=1}^N\left|f(n)-\sum_{j=1}^kc_je^{2\pi i\alpha_jn}\right|<\epsilon$$
  \item An increasing sequence $(n_k)_{k\in\N}$ in $\N$ is a \emph{rational Besicovitch sequence} if the indicator function
      $1_A$ of the set $A=\{n_k:k\in\N\}$ is rational Besicovitch almost periodic and $A$ has positive upper density.
\end{itemize}
\end{definition}

It follows from (the proof of) Proposition \ref{prop_apexample} that the sequence $(n_k)_{k\in\N}$ enumerating the squarefree numbers is rational Besicovitch.
\begin{theorem}\label{theorem_rationalbesicovitch}
  Let $u:\N\mapsto\T$ and assume that for every $\alpha\in\Q$, the sequence $n\mapsto u(n)+n\alpha$ is uniformly distributed.
  Let $(n_k)_{k\in\N}$ be a rational Besicovitch sequence.
  Then the sequence $k\mapsto u(n_k)$ is uniformly distributed in $\T$.
\end{theorem}


Theorem \ref{theorem_rationalbesicovitch} immediately implies that if $(n_k)$ is the increasing enumeration of the set of squarefree numbers then, for any irrational $\alpha$, the sequence $k\mapsto n_k\alpha\bmod1$ is uniformly distributed.
A more sophisticated application involves Boshernitzan's theorem (Theorem \ref{theorem_Boshernitzan}).
It is easy to see that a function $f:\R\to\R$ satisfies \eqref{eq_boshernitzan} for every $p\in\Q[x]$ if and only if $f(x)+x\alpha$ does, for a rational $\alpha$.
Therefore, we obtain the following corollary of Theorem \ref{theorem_Boshernitzan}:
\begin{corollary}\label{cor_besibosher}
  Let $f:\R\to\R$ be subpolynomial and a member of a Hardy field.
  Let $(n_k)$ be a rational Besicovitch sequence.
  Assume that condition \eqref{eq_boshernitzan} holds for every $p\in\Q[x]$.
  Then the sequence $\big(f(n_k)\bmod1\big)_{k\in\N}$ is uniformly distributed.
\end{corollary}
\section{
Some additional applications of \DT}\label{sec_othergroups}
In this section we demonstrate the versatility of \DT{} by providing three more applications of Theorem \ref{theorem_vdctrickcontinuous}.

\subsection{A Weyl type theorem in $\A/\Q$}\label{sec_adeles}

In this subsection, with the help of (a special case of) Theorem \ref{thm_intro_vdCgeneral}, we establish an ``adelic'' version of Weyl's theorem on equidistribution of polynomials.

Fix a prime $p$.
Given $x,y\in\Q$, write $x-y=p^n\cdot\tfrac ab$ where $a,b,n\in\Z$, $b>0$ and both $a$ and $b$ are coprime with $p$.
Let $d_p(x,y)=p^{-n}$.
It is not hard to check that $d_p$ is a metric on $\Q$.
We denote by $\Q_p$ the completion of $\Q$ with respect to this metric and extend $d_p$ to $\Q_p$.
A more concrete description of $\Q_p$ is
$$\Q_p=\left\{\sum_{i=N}^\infty x_ip^i:x_i\in\{0,\dots,p-1\};\ N\in\Z\right\}$$
We let $\Z_p\subset\Q_p$ be the set $\{n\in\Q_p:d_p(n,0)\leq1\}$.
Equivalently, $\Z_p$ consists of sums $\sum_{i=N}^\infty x_ip^i$ where $N\geq0$.
Observe that $\Z\subset\Z_p$.

Next, let $\A$ denote the additive group of the adeles:
$$\A=\left\{(a_\infty,a_2,a_3,a_5,\dots)\in\left(\R\times\prod_{p\text{ prime}}\Q_p\right):a_p\in\Z_p\text{ for all but finitely many }p\right\}$$
If $x\in\Q$ and $p$ is a prime which does not divide the denominator of $x$ then $x\in\Z_p$; in particular, $x\in\Z_p$ for all but finitely many $p$.
Therefore, one can embed $\Q$ in $\A$ by identifying $x\in\Q$ with $(x,x,x,\dots)\in\A$.
Given $x=\sum_{i=N}^\infty x_ip^i\in\Q_p$ with $x_i\in\{0,\dots,p-1\}$ we denote by
\begin{equation}\label{eq_fractionalpartpadic}f_p(x)=\sum_{i=N}^{-1} x_ip^i
\end{equation}
the $p$-fractional part of $x$ (with the convention that $f_p(x)=0$ when $N>-1$).
Therefore $x-f_p(x)\in\Z_p$ for any $x\in\Q_p$.
Observe that $f_p(x)$ is in $\Q$ for any $x\in\Q_p$ since \eqref{eq_fractionalpartpadic} is a finite sum.
Moreover, the denominator of $f_p(x)$ is a power of $p$; it follows that $f_p(x)\in\Z_q$ for any $p\neq q$ and $x\in\Q_p$.
Given $x\in\A$, the sum $\tilde f(x)=\sum_pf_p(x_p)\in\Q$ has only finitely many non-zero terms and hence is well defined. Note that each $p$-adic coordinate of $x-\tilde f(x)$ is in $\Z_p$ and therefore $x-\tilde f(x)\in\R\times\prod_p\Z_p$.
Finally, let
\begin{equation}\label{eq_fractionaladeles}
  f(x)=\big(\tilde f(x)+\lfloor x_\infty-\tilde f(x)\rfloor\big)\in\Q
\end{equation}
and observe that $x-f(x)\in[0,1)\times\prod_p\Z_p$.

The topology on $\A$ is generated by a basis of open sets of the form
$$U_\infty\times\prod_{p\in F}U_p\times\prod_{p\notin F}\Z_p$$
where $U_\infty\subset\R$ is open, $F$ is a finite set of primes and $U_p\subset\Q_p$ is an open set.

The following lemma is well known; we give a proof for the sake of completeness.
\begin{lemma}
  The subgroup $\Q\subset\A$ is discrete and co-compact.
\end{lemma}
\begin{proof}
  Let $x\in\Q$ be an arbitrary rational number. 
On the one hand, observe that 
the set
$$U_x=(x-1,x+1)\times\prod_p\left(x+\Z_p\right)\subset\A$$
is open, because $x\in\Z_p$ for all but finitely many primes $p$ (those which divide the denominator of $x$).
On the other hand, for any rational $y\in U_x$, the difference $y-x$ belongs to every $\Z_p$ and hence to $\Z$. However, that difference is also in $(-1,1)$, and so it must be $0$.
We just showed that $x$ is the only rational in $U_x$, and because $x\in\Q$ was arbitrary it follows that $\Q$ is a discrete subset of $\A$.

Next we focus on the quotient group $K:=\A/\Q$.
Let $\tilde K:=[0,1)\times\prod_p\Z_p\subset\A$.
We define a group structure in $\tilde K$ as follows: given $x,y\in\tilde K$, let $x+y:=(x+y)-f(x+y)$, where the sums and the subtraction on the right are the (usual) pointwise operations in $\A$, and $f:\A\to\Q$ is defined by \eqref{eq_fractionaladeles}.
Note that this group operation is different from the pointwise addition.

We claim that $K$ is isomorphic to $\tilde K$.
To show this, let $\phi:\A\to\tilde K$ be the map $\phi(x)=x-f(x)\in\tilde K\subset\A$.
First we show that $\phi$ is a group homomorphism.
Indeed, for any $x,y\in\A$, the difference (in $\A$) between $\phi(x+y)$ and $\phi(x)+\phi(y)$
is a rational number which belongs to the set $\tilde K-\tilde K=(-1,1)\times\prod_p\Z_p$, and $0$ is the only such number.

Next, observe that $\phi$ is surjective. Indeed, for $x\in\tilde K\subset\A$ we have $f(x)=0$ and hence $\phi(x)=x$.
To see that the kernel of $\phi$ is $\Q$, notice that $f(x)\in\Q$ for any $x\in\A$, and so, on the one hand,  $x=f(x)\Rightarrow x\in\Q$, and on the other hand, $x\in\Q\Rightarrow\phi(x)\in\Q\cap\tilde K=\{0\}$.

Finally, observe that $\phi$ is not quite a homeomorphism, because the open set $[0,1/2)\times\prod_p\Z_p\subset\tilde K$ can not be written as $\phi(U)$ where $U\subset\A$ is an open set invariant under $\Q$.
In fact, if we endow $K\cong\tilde K$ with the topology consisting of the open sets $\phi(U)$ with $U\subset\A$ open and invariant under $\Q$, we see that the topology on $K$ is the same as $\T\times\prod_p\Z_p$ (but keep in mind that the group structure in $K$ is \emph{not} the product (or pointwise) structure).
\end{proof}


A function $g:\Q\to K$ is a \emph{polynomial} if it takes the form $g(x)=\sum_{j=0}^m\alpha_jx^j$ for some $\alpha_j\in\A$, $j=0,\dots,m$.
We say that the polynomial $g$ \emph{has an irrational coefficient} if for some $j>0$ one has $\alpha_j\in\A\setminus\Q$.
\begin{theorem}\label{thm_adelespolequidistribution}
  Let $g_1,\dots,g_d:\Q\to K$ be polynomials.
  If for any $(n_1,\dots,n_d)\in\Q^d\setminus\{0\}$ the sum $\sum n_jg_j:\Q\to K$ is not a constant map, then the sequence $(g_1,\dots,g_d):\Q\to K^d$ is well distributed.
\end{theorem}
In particular, this sequence has a dense image; a special case of this fact was established in \cite{Corwin_Pfeffer95} using different methods.
A related result, concerning a kind of uniform distribution of polynomial sequences in $\A/\Q$ and more general settings, was obtained by D. Cantor in \cite{Cantor65}.
Cantor  used significantly more sophisticated machinery from number theory.

Our proof of Theorem \ref{thm_adelespolequidistribution} is very similar to the classical proof of (multidimensional version of) Weyl's theorem using \DT{} and rests on the Weyl criterion, Theorem \ref{theorem_weylgeneral}.
We state in the following proposition the precise version of Weyl's criterion we will need.
\begin{proposition}\label{prop_adelesweyl}
  A sequence $f:\Q\to K^d$ is well distributed if and only if for any F\o lner sequence $(F_N)_{N\in\N}$ in $\Q$ and any non-constant character $\chi: K^d\to\C$ we have:
  $$\lim_{N\to\infty}\frac1{|F_N|}\sum_{u\in F_N}\chi\big(f(x)\big)=0$$
\end{proposition}
\begin{proof}
  Apply Theorem \ref{theorem_weylgeneral} with $G=\Q$ with the discrete topology, $\lambda$ being the counting measure on $\Q$ and the compact group ($K$ in the statement of Theorem \ref{thm_adelespolequidistribution}) being $K^d$.
\end{proof}

The next step is to give a description of the characters of $K^d$.
It is a classical fact that the Pontryagin dual of $K$ is $\Q$ (with the discrete topology).
More precisely, with any point $r\in\Q$ we associate the character $\chi_r:K\to\C$ described by
$$\chi_r:u\mapsto\exp\Big(2\pi i\big(\tilde f(ru)-nu_\infty\big)\Big)$$
More generally, the Pontryagin dual of $K^d$ is $\Q^d$ and, for $r=(r_1,\dots,r_d)\in\Q^d$ we define the character $\chi_r:K^d\to\C$ by the formula
$$\chi_r:u=(u_1,\dots,u_d)\mapsto\prod_{j=1}^d\chi_{r_j}(u_j)$$

Observe that, if we denote $\chi_1$ by simply $\chi$, then $\chi_r(u)=\chi(ru)$ for any $r\in\Q$ and $u\in\A$.

\begin{proof}[Proof of Theorem \ref{thm_adelespolequidistribution}]
Let $r\in\Q^d\setminus\{0\}$ be arbitrary and let $(F_N)_{N\in\N}$ be an arbitrary F\o lner sequence for the additive group $(\Q,+)$.
In view of Proposition \ref{prop_adelesweyl}, we need to show that
\begin{equation}\label{eq_adelesequidistribution}
  \lim_{N\to\infty}\frac1{|F_N|}\sum_{x\in F_N}\chi_r\big(g_1(x),\dots,g_d(x)\big)=0
\end{equation}
Expanding the definition of $\chi_r$ we get

\begin{equation}
\label{eq_adelestheoremproof1}
\chi_r\big(g_1(x),\dots,g_d(x)\big)=\prod_{j=1}^d\chi_{r_j}\big(g_j(x)\big)= \chi\left(\sum_{j=1}^dr_jg_j(x)\right)
\end{equation}

Without loss of generality, we may (and will) assume that $g_i(0)=0$.
We proceed by induction on the highest degree of the polynomials $g_1,\dots,g_d$.
Assume first that the highest degree is $1$, so that
$g_j(x)=\alpha_jx$ where all $\alpha_j\in\A$ and, for every $r=(r_1,\dots,r_d)\in\Q^d\setminus\{0\}$, we have $\alpha:=\sum_{j=1}^dr_j\alpha_j\notin\Q$.

In view of \eqref{eq_adelesequidistribution} and \eqref{eq_adelestheoremproof1} it suffices to show that
$$\lim_{N\to\infty}\frac1{|F_N|}\sum_{x\in F_N}\chi(\alpha x)=0$$
for any $\alpha\in\A\setminus\Q$ and any additive F\o lner sequence $(F_N)_{N\in\N}$ in $\Q$.
Since $\alpha\in\A\setminus\Q$, the projection $\alpha+\Q\in K$ is nonzero.
Therefore there exists some character of $K$ which does not vanish at $\alpha$.
Since $K$ is the Pontryagin dual of $\Q$, this means that $\chi(\alpha y)\neq1$ for some $y\in\Q$.

Next take $\epsilon>0$ and let $N\in\N$ be such that
$$\frac{\big|(F_N+y)\bigtriangleup F_N\big|}{|F_N|}\leq\big|1-\chi(\alpha y)\big|\epsilon$$
Then, on the one hand, we have
$$\left|\frac1{|F_N|}\sum_{x\in F_N}\chi(\alpha x)-\frac1{|F_N|}\sum_{x\in F_N+y}\chi(\alpha x)\right|\leq\frac{\big|(F_N+y)\bigtriangleup F_N\big|}{|F_N|}<\big|1-\chi(\alpha y)\big|\epsilon,$$
and, on the other hand, we have
$$\frac1{|F_N|}\sum_{x\in F_N+y}\chi(\alpha x)=\frac1{|F_N|}\sum_{x\in F_N}\chi\big(\alpha(x+y)\big)=\chi(\alpha y)\frac1{|F_N|}\sum_{x\in F_N}\chi(\alpha x)$$
Putting things together we conclude that
$$\left|\big(1-\chi(\alpha y)\big)\frac1{|F_N|}\sum_{x\in F_N}\chi(\alpha x)\right|\leq\big|1-\chi(\alpha y)\big|\epsilon$$
which finally implies that
$$\left|\frac1{|F_N|}\sum_{x\in F_N}\chi(\alpha x)\right|<\epsilon$$
This finishes the proof of the theorem in the case when all the polynomials $g_1,\dots,g_d$ have degree $1$.

Next assume that $s:=\max\{\deg g_1,\dots,\deg g_d\}>1$ and that the result holds for $s-1$.
We will apply \DT{}, Theorem \ref{thm_intro_vdCgeneral}, with the role of $G$ being played by $\Q$ endowed with the discrete topology and the counting Haar measure.
For any $h\in\Q$ and $j\in\{1,\dots,d\}$, the difference $x\mapsto g_j(x+h)-g_j(x)$ is a polynomial of degree at most $s-1$.
It follows from the induction hypothesis that
$$g_h(x):=g(x+h)-g(x)=\big(g_1(x+h)-g_1(x),\dots,g_d(x+h)-g_d(x)\big)$$
is uniformly distributed in $K^d$ for any $h\in\Q\setminus\{0\}$.
The result now follows from Theorem \ref{thm_intro_vdCgeneral}.
\end{proof}

The technique employed in this section works in more general settings. For instance, one can obtain a similar theorem for polynomials $p:K\to\A_K/K$, where $K$ is a number field and $\A_K$ is the group of adeles over $K$.
These results also extend to the case of polynomials with several variables.
In order to keep the exposition short, we do not pursue these possibilities any further.
\subsection{An application to the $\{x+y,xy\}$ problem}\label{sec_x+yxy}
A long standing open question in Ramsey theory asks whether for any finite partition $\N=C_1\cup\cdots\cup C_r$ of the
natural numbers, there exists a cell $C_i$ of the partition and a pair $(x,y)\in\N^2$ such that
$\{x,y,x+y,xy\}\subset C_i$.
The answer is not known even if one only asks for $\{x+y,xy\}\subset C_i$ and $(x,y)\neq(2,2)$.
In \cite{Bergelson_Moreira15} a similar question pertaining to partitions of fields was answered positively.
In particular, the following result was proved:
\begin{theorem}\label{theorem_x+y,xy}
  For every infinite field $K$ and every finite partition $K=C_1\cup\dots\cup C_r$ there exists a cell $C_i$ and
  infinitely many $x,y\in K$ such that $\{x,x+y,xy\}\subset C_i$.
\end{theorem}
The proof of Theorem \ref{theorem_x+y,xy} in \cite{Bergelson_Moreira15} is based on ergodic-theoretic techniques.
At its core lies a new ergodic theorem, whose proof hinges on a new application of the following form of \DT.
\begin{proposition}\label{prop_vdCfields}
  Let $H$ be a Hilbert space, let $K$ be a countable field, let $(F_N)_{N\in\N}$ be a F\o lner sequence in the multiplicative group $K^*$ and let $u:K^*\to H$ be a bounded function.
  Assume that
  $$\lim_{D\to\infty}\frac1{|F_D|}\sum_{d\in F_D}\limsup_{N\to\infty}\left|\frac1{|F_N|}\sum_{x\in F_N}\big\langle u(xd),u(x)\big\rangle\right|=0$$
  Then
  $$\lim_{N\to\infty}\frac1{|F_N|}\sum_{x\in F_N}u(x)=0$$
\end{proposition}
\begin{proof}
  Apply Theorem \ref{theorem_vdctrickcontinuous} with $G$ being the multiplicative group $K^*$, endowed with the discrete topology, and $\lambda$ being the counting measure.
\end{proof}
\begin{theorem}
  Let $K$ be a countable infinite field, let $(A_x)_{x\in K}$ and $(M_y)_{y\in K^*}$ be unitary representations, respectively, of the additive and multiplicative groups
  $(K,+)$ and $(K^*,\times)$ on the same Hilbert space $H$ and satisfying the distributivity law
  \begin{equation}\label{eq_distributivitylaw}
  M_yA_x=A_{xy}M_y\qquad\forall x,y\in K,~y\neq0
  \end{equation}
  Then for every sequence $(F_N)_{N\in\N}$ of finite subsets of $K^*$ which is simultaneously a F\o lner sequence for both\footnote{Such sequences exist in any countable field $K$ (see Proposition 2.4 in \cite{Bergelson_Moreira15}).}
  $(K,+)$ and $(K^*,\times)$, and for every $f\in H$ we have
  $$\lim_{N\to\infty}\frac1{|F_N|}\sum_{x\in F_N}M_xA_xf=Pf$$
  where $P$ is the orthogonal projection onto the invariant subspace $\{h\in H:A_xh=M_yh=h~\forall x,y\in K, y\neq0\}$.
\end{theorem}

\begin{proof}
Each $f\in H$ has a unique representation of the form $f=Pf+g$, where $g=f-Pf$ is orthogonal to the invariant subspace.
It is clear that $M_xA_xPf=Pf$, so it suffices to show that the Ces\`aro limit of $M_xA_xg$ is $0$.
  The idea is to apply Proposition \ref{prop_vdCfields} to the function $u:(K^*,\times)\to H$ defined by $u(x)=M_xA_xg$.

  Let $d\in K^*$.
  Using \eqref{eq_distributivitylaw} and the fact that the representations are unitary we have
  $$\langle u(xd),u(x)\rangle=\langle M_{xd}A_{xd}g,M_xA_xg\rangle=\langle A_{-x}M_dA_{xd}g,g\rangle=\langle
  A_{x(d-1/d)}g,M_{1/d}g\rangle$$
  Observe that $x\mapsto x(d-1/d)$ is an automorphism of the group $(K,+)$ for every $d\neq1$.
  Applying twice the mean ergodic theorem (see, for instance, \cite[Theorem 5.5]{Bergelson06}) we have
  $$\lim_{D\to\infty}\frac1{|F_D|}\sum_{d\in F_D}\lim_{N\to\infty}\frac1{|F_N|}\sum_{x\in F_N}\langle
  u(dx),u(x)\rangle=\langle P_Ag,P_Mg\rangle$$
  where $P_A$ is the orthogonal projection onto the additively invariant subspace $\{h\in H:A_xh=h~\forall x\in K\}$ and
  $P_M$ is the orthogonal projection onto the multiplicatively invariant subspace $\{h\in H:M_xh=h~\forall x\in K^*\}$.
  One can show (see \cite[Lemma 3.1]{Bergelson_Moreira15}) that $P_M$ and $P_A$ commute.
  Therefore $\langle P_Ag,P_Mg\rangle=\|P_AP_Mg\|^2$, but $P_AP_Mg=P_MP_Ag$ is invariant under both the additive and
  the multiplicative representation.
  This implies that $P_AP_Mg=Pg=0$.
  We can now invoke Proposition \ref{prop_vdCfields} and conclude that
  $$\lim_{N\to\infty}\frac1{|F_N|}\sum_{x\in F_N}M_xA_xg=0$$
  which finishes the proof.
\end{proof}

\subsection{A non-linear multiple convergence theorem for ergodic $\R$-actions}\label{sec_ractions}
In this subsection we demonstrate the power of \DT{} (in the form of Theorem \ref{theorem_vdctrickcontinuous} with $G=\R$) in yet another situation.
Namely, we obtain a rather general kind of non-linear ergodic theorem involving several commuting measure preserving $\R$-actions.
\begin{definition}
  An increasing function $p\in C^2(b,\infty)$ defined in some interval $(b,\infty)$ is called \emph{admissible} if it satisfies the following:
  \begin{enumerate}
    \item The derivative $p'$ is 
    monotone.
    \item For every $c\in(0,1)$ there exist $\epsilon<1$ such that for large enough $\tau\in\R$ we have $\epsilon<p(c\tau)/p(\tau)<1-\epsilon$.
  \end{enumerate}
\end{definition}

Observe that the inverse function $p^{-1}$ of an \resfunc{} $p$ exists in some interval $(\tilde b,\infty)$ and is also an \resfunc.
Moreover, the composition $p_1\circ p_2$ of two \resfunc s is again an \resfunc.
In other words, the class of (germs of) \resfunc s forms a group under composition of functions.
Observe that if $p\in C^2(b,\infty)$ has monotone derivative and if for every $\epsilon>0$ there exist a positive $\alpha$ such that
$$\lim_{s\to\infty}\frac{s^\alpha}{f(s)}=\lim_{s\to\infty}\frac{f(s)}{s^{\alpha+\epsilon}}=0,$$
then $f$ is an \resfunc.
This observation gives the following examples.
\begin{example}\label{example_funcions}
  The following are \resfunc:
  \begin{enumerate}
    \item $p(s)=\sum_{i=1}^kc_is^{\alpha_i}$, where $c_i,\alpha_i\in\R$, $c_k,\alpha_k>0$, $\alpha_1<\cdots<\alpha_k$.
    \item  $p(s)=s^\alpha\log^\beta s$, where $\alpha>0$ and $\beta\in\R$.
    \item $p(s)=s^\alpha\big(\cos(\log^\beta s)+2\big)$, where $\alpha>0$ and $\beta<1$.
    \item $p(s)=s^\alpha\left(1+\tfrac{\sin s}{\log s}\right)$ for $\alpha>0$.
  \end{enumerate}
\end{example}

Observe that (3) and (4) in Example \ref{example_funcions} do not belong to any Hardy field.

Here now is the formulation of the first theorem of this subsection.
\begin{theorem}\label{thm_temperedcontconvergence}
  Let $k,\ell$ be positive integers and let $(X,{\mathcal B},\mu)$ be a probability space. For each $i\in\{1,\dots,k\}$ and $j\in\{1,\dots,\ell\}$ let $(T_{i,j}^s)_{s\in\R}:X\to X$ be a continuous ergodic\footnote{An action $(T_g)_{g\in G}$ of a group $G$ on a probability space $(X,{\mathcal B},\mu)$ by measurable maps is \emph{ergodic} if any set $B\in{\mathcal B}$ satisfying $T_g^{-1}(B)=B$ for every $g\in G$ is trivial, in the sense that $\mu(B)\in\{0,1\}$.} measure preserving $\R$-action and let $p_{i,j}:\R\to\R$ be \resfunc s.
  Assume that all the $T_{i,j}$ commute (in the sense that $T_{i,j}^sT_{\tilde i,\tilde j}^{\tilde s}=T_{\tilde i,\tilde j}^{\tilde s}T_{i,j}^s$ for any $i,\tilde i\in\{1,\dots,k\}$, $j,\tilde j\in\{1,\dots,\ell\}$ and $s,\tilde s\in\R$) and that $\lim p_{i,j}(s)/p_{\tilde i,\tilde j}(s)=0$ whenever $i>\tilde i$ or ($i=\tilde i$ and $j>\tilde j$).

  Then for any $f_1,\dots,f_k\in L^\infty(X)$ we have
  \begin{equation}\label{eq_thm_temperedcontconvergence}\lim_{\tau\to\infty}\frac1\tau\int_0^\tau \prod_{i=1}^k\left[\left(\prod_{j=1}^\ell T_{i,j}^{p_{i,j}(s)}\right)f_i\right]\dd s=\prod_{i=1}^k\int_Xf_i\dd\mu\qquad\text{ in }L^2(X)\end{equation}
\end{theorem}


The strategy of the proof of Theorem \ref{thm_temperedcontconvergence} is to apply a change of variables and then use (an appropriate version of) \DT.
We will need the following technical lemma, asserting essentially that a change
of variables involving an \resfunc{} yields an equivalent method of summation.
Similar ``change of variable'' lemmas appeared for instance in \cite[Lemma 7.7]{Bergelson_Leibman_Son14} and \cite[Lemma 2.2]{Austin12}, but we were unable to find the precise formulation we need in the literature.

\begin{lemma}\label{lemma_changeofvariable}
  Let $H$ be a Hilbert space, let $a:\R\to H$ be a continuous and bounded function and let $\sigma\in C^2(0,\infty)$ be an \resfunc{}. 
  Then
  \begin{equation}\label{eq_lemma_changeofvariable}
    \lim_{\tau\to\infty}\frac1\tau \int_0^\tau a(s)-a\big(\sigma(s)\big)\dd s=0\qquad\text{ in norm,}
  \end{equation}
  where the integral is understood in the sense of Bochner.
\end{lemma}


\begin{proof}
Assume, without loss of generality, that $\|a(s)\|\leq1$ for every $s$.
We will use the notation $A_{[x,y]}:=\frac1{y-x}\int_x^ya(s)\dd s$.
Assume that $\lim_{\tau\to\infty}A_{[0,\tau]}$ exists and call it $L$.
Let $f$ be the inverse of $\sigma$.
Since the contribution of the initial segment $[0,\tau_0]$ is asymptotically negligible when one considers long range Ces\`aro averages
, we may and will assume, without loss of generality, that $f$ is defined on $(0,\infty)$.

We claim that for any $c\in(0,1)$ and any function $u:\R\to\R$ satisfying $0\leq u(s)\leq\sigma\big(cf(s)\big)$ we have
\begin{equation}\label{eq_lemma_changeofvariable3}
  \lim_{s\to\infty}A_{[u(s),s]}=L
\end{equation}
Indeed,
\begin{equation}\label{eq_affinecombination} A_{[u(s),s]}=\frac s{s-u(s)}A_{[0,s]}-\frac{u(s)}{s-u(s)}A_{[0,u(s)]}
\end{equation}
Since $\sigma(c\tau)\leq\delta\sigma(\tau)$ for some $\delta$ depending on $c$ (but not on $\tau$), taking $\tau=f(s)$ we get that $u(s)\leq\delta s$.
It follows that the coefficients of the linear combination in \eqref{eq_affinecombination} add up to $1$ and stay bounded as $s\to\infty$.
Since both $A_{[0,s]}$ and $A_{[0,u(s)]}$ approach $L$ as $s\to\infty$, this proves the claim \eqref{eq_lemma_changeofvariable3}.

We now carry out some computations.
Making the change of variables $s=f(t)$ we have
$$\int_0^\tau a\big(\sigma(s)\big)\dd s=\int_0^{\sigma(\tau)}a(t)f'(t)\dd t=\int_0^{\sigma(\tau)}a(t)\int_0^tf''(u)\dd u\dd t$$
Since both functions $a$ and $f''$ are bounded in the relevant intervals, we may switch the order of integrals and obtain
\begin{align}\label{eq_lemma_changeofvariable1}
   \int_0^\tau a\big(\sigma(s)\big)\dd s& =\int_0^{\sigma(\tau)}f''(u)\int_u^{\sigma(\tau)}a(t)\dd t\dd u \notag\\
   & =\int_0^{\sigma(\tau)}f''(u)\big(\sigma(\tau)-u\big)A_{[u,\sigma(\tau)]}\dd u
\end{align}
For future use, we now compute the following integral

\begin{equation}
\label{eq_lemma_changeofvariable4}
\begin{split}
\int_x^yf''(u)\big(\sigma(\tau)-u\big)\dd u&=f'(u)\big(\sigma(\tau)-u\big)\Big|_x^y+\int_x^yf'(u)\dd u
\\ &=f'(y)\big(\sigma(\tau)-y\big)-f'(x)\big(\sigma(\tau)-x\big)+f(y)-f(x)
\end{split}
\end{equation}

Let $\epsilon>0$, let $c\in(1-\epsilon,1)$ and let $M=\sup_\tau\sigma'(\tau)/\sigma'(c\tau)$.
Observe that $M$ is finite, as a consequence of L'H\^opital's rule (and of the fact that $\sigma$ is an \resfunc).
 We will show that 
 we can truncate the integral in \eqref{eq_lemma_changeofvariable1} at $u=\sigma(c\tau)$, at the cost of an error of size at most $M\epsilon\tau$.

 Since $\|a(s)\|\leq1$, we have $\|A_{[x,y]}\|\leq1$ for every $x<y$.
 Also, since $f$ is an \resfunc, the sign of $f''$ doesn't change.
 Therefore, using \eqref{eq_lemma_changeofvariable4} we obtain
 $$\begin{aligned}
    \left\|\int_{\sigma(c\tau)}^{\sigma(\tau)}f''(u)\big(\sigma(\tau)-u\big)A_{[u,\sigma(\tau)]}\dd u\right\|\quad \leq\quad\left|\int_{\sigma(c\tau)}^{\sigma(\tau)}f''(u)\big(\sigma(\tau)-u\big)\dd u\right|&\\
   = \Big|-f'\big(\sigma(c\tau)\big)\big(\sigma(\tau)-\sigma(c\tau)\big)+f\big(\sigma(\tau)\big)-f\big(\sigma(c\tau)\big) \Big| &\\
   = \left|\frac{\sigma(\tau)-\sigma(c\tau)}{-\sigma'(c\tau)}+\tau(1-c) \right|&\\
   = \tau(1-c)\left|1-\frac{\sigma'(b)}{\sigma'(c\tau)}\right|&
 \end{aligned}$$
for some $b\in (c\tau,\tau)$.
If $\sigma'$ is decreasing, the last expression is bounded by $\tau\epsilon$.
If $\sigma'$ is incresing, it is bounded by $\tau\epsilon(M-1)$.
In either case we obtain
\begin{equation}\label{eq_lemma_changeofvariable2}
  \left\|\frac1\tau\int_{\sigma(c\tau)}^{\sigma(\tau)}f''(u)\big(\sigma(\tau)-u\big)A_{[u,\sigma(\tau)]}\dd u\right\|\leq M\epsilon
\end{equation}

Putting \eqref{eq_lemma_changeofvariable1} and \eqref{eq_lemma_changeofvariable2} together we obtain
$$\frac1\tau\int_0^\tau a\big(\sigma(s)\big)\dd s=\frac1\tau\int_0^{\sigma(c\tau)}f''(u)\big(\sigma(\tau)-u\big)A_{[u,\sigma(\tau)]}\dd u+R_1$$
where $R_1=R_1(\tau)\in H$ satisfies $\|R_1\|\leq M\epsilon$.

On the other hand, for $u\in\big(0,\sigma(c\tau)\big)$, using \eqref{eq_lemma_changeofvariable3} with $s=\phi(\tau)$ (and hence $\tau=f(s)$) and $\tau$ large enough, we obtain that $\|A_{[u,\sigma(\tau)]}-L\|<\epsilon$.
Therefore
$$\frac1\tau\int_0^\tau a\big(\sigma(s)\big)\dd s=\frac1\tau\int_0^{\sigma(c\tau)}f''(u)\big(\sigma(\tau)-u\big)L\dd u+R_2+R_1$$
for some $R_2=R_2(\tau)\in H$ satisfying $\|R_2\|\leq\epsilon$.
Calculating the last integral using \eqref{eq_lemma_changeofvariable4}, taking the limit as $\tau\to\infty$ and letting $\epsilon\to0$ we conclude that
$$\lim_{\tau\to\infty}\frac1\tau\int_0^\tau a\big(\sigma(s)\big)\dd s=L$$


\end{proof}

We will need yet another lemma:

\begin{lemma}\label{lemma_continuousconvergencevanishing}
   Let $(X,{\mathcal B},\mu,(T^s)_{s\in\R})$ be an ergodic continuous measure preserving $\R$-action and let
   $b:\R\to\R$ be a continuous function such that $\lim_{s\to\infty}b(s)=0$.
   Let $f\in L^1(X)$ and let $(g_s)_{s\in\R}$ be a family of functions in $L^\infty(X)$ such that
   $\sup_s\|g_s\|_{L^\infty}<\infty$.
   Then
   $$\lim_{\tau \to\infty}\frac1\tau \int_0^\tau  \int_Xf\cdot g_s\dd\mu\dd s=\lim_{\tau \to\infty}\frac1\tau
   \int_0^\tau  \int_XT^{b(s)}f\cdot g_s\dd\mu\dd s$$
   in the sense that the existence of one limit implies the existence of the other and the equality.
\end{lemma}

\begin{proof}
  It suffices to show that
  $$\lim_{s\to\infty}\int_X\big(f-T^{b(s)}f\big)\cdot g_s\dd\mu=0$$
  Let $M=\sup_s\|g_s\|_{L^\infty}$.
  Then we have
  $$\left|\int_X\big(f-T^{b(s)}f\big)\cdot g_s\dd\mu\right|\leq M\left\|f-T^{b(s)}f\right\|_{L^1}$$
  Since $b(s)\to0$ and $T$ is a continuous action, we deduce that $T^{b(s)}f\to f$ in $L^1(X)$ and hence
  $$\limsup_{s\to\infty}\left|\int_X\big(f-T^{b(s)}f\big)\cdot g_s\dd\mu\right|\leq
  M\lim_{s\to\infty}\left\|f-T^{b(s)}f\right\|_{L^1}=0$$
\end{proof}

The next proposition uses \DT.
The exact form we need follows from Theorem \ref{theorem_vdctrickcontinuous} by letting $G=\R$ with its usual topology and the Lebesgue Haar measure $\lambda$, and letting $(F_N)_{N\in\N}$ be any sequence of intervals $[0,b_N]$ with $b_N\to\infty$.

\begin{proposition}\label{prop_convergencecontinuous}
Let $k,\ell$ be positive integers and let $(X,{\mathcal B},\mu)$ be a probability space. For each $i\in\{1,\dots,k\}$ and $j\in\{1,\dots,\ell\}$ let $(T_{i,j}^s)_{s\in\R}:X\to X$ be a continuous ergodic measure preserving $\R$-action and let $q_{i,j}:\R\to\R$ be an \resfunc.
  Assume that all the $T_{i,j}$ commute, that $\lim q_{1,1}'(s)=1$ as $s\to\infty$ and that $\lim q_{i,j}'(s)=0$ for each $(i,j)\neq(1,1)$.

  If either $k=1$, or $k>1$ and Theorem \ref{thm_temperedcontconvergence} holds for $k-1$, then for any $f_1,\dots,f_k\in L^\infty(X)$ with $\int_Xf_1d\mu=0$ we have
  \begin{equation*}\lim_{\tau\to\infty}\frac1\tau\int_0^\tau \prod_{i=1}^k\left[\left(\prod_{j=1}^\ell T_{i,j}^{q_{i,j}(s)}\right)f_i\right]\dd s=0\end{equation*}

\end{proposition}
\begin{proof}

 With Theorem \ref{theorem_vdctrickcontinuous} in mind (with $G=\R$ and $(F_N)_{N\in\N}$ being any sequence of intervals $[0,\tau_N]$ with $\tau_N\to\infty$), let $a(s)=\prod_{i=1}^k\left[\left(\prod_{j=1}^\ell T_{i,j}^{q_{i,j}(s)}\right)f_i\right]$ and $h\in\R$.
  We have
  \begin{eqnarray*}
    \langle a(s+h),a(s)\rangle&=&\int_X\prod_{i=1}^k\left[\left(\prod_{j=1}^\ell T_{i,j}^{q_{i,j}(s+h)}\right)f_i\cdot \left(\prod_{j=1}^\ell T_{i,j}^{q_{i,j}(s)}\right)f_i\right]\dd\mu\\&=&
    \int_X\prod_{i=1}^k\left[\left(\prod_{j=1}^\ell T_{i,j}^{q_{i,j}(s)}\right)\left[f_i\cdot \left(\prod_{j=1}^\ell T_{i,j}^{q_{i,j}(s+h)-q_{i,j}(s)}\right)f_i\right]\right]\dd\mu
  \end{eqnarray*}
  The mean value theorem implies that $\lim_s q_{i,j}(s+h)-q_{i,j}(s)=0$ for each fixed $h$ and every $(i,j)\neq(1,1)$, and also that $\lim_sq_{1,1}(s+h)-q_{1,1}(s)-h=0$ for any fixed $h$.
  In view of Lemma \ref{lemma_continuousconvergencevanishing} we have
  \begin{multline}\label{eq_continuouscommuting1}
     \lim_{\tau \to\infty}\frac1\tau \int_0^\tau \langle a(s+h),a(s)\rangle\dd s=\\
    \lim_{\tau \to\infty}\frac1\tau
  \int_0^\tau \int_X\left(\prod_{j=1}^\ell T_{1,j}^{q_{1,j}(s)}\right)\big(f_1\cdot T_{1,1}^hf_1\big)\cdot\prod_{i=2}^k\left(\prod_{j=1}^\ell T_{i,j}^{q_{i,j}(s)}\right)\big(f_i^2\big)\dd\mu\dd s
  \end{multline}
  If $k=1$ the expression \eqref{eq_continuouscommuting1} reduces to
  $$\lim_{\tau \to\infty}\frac1\tau
  \int_0^\tau \int_X\left(\prod_{j=1}^\ell T_{1,j}^{q_{1,j}(s)}\right)\big(f_1\cdot T_{1,1}^hf_1\big)\dd\mu\dd s$$
  Using the fact that each $T_{1,j}^{q_{1,j}(s)}$ preserves $\mu$, we have
  \begin{equation}\label{eq_continuouscommuting3} \lim_{\tau \to\infty}\frac1\tau \int_0^\tau \langle a(s+h),a(s)\rangle\dd s=\int_X\big(f_1\cdot
  T_{1,1}^hf_1\big)\dd\mu
  \end{equation}
  It follows from the mean ergodic theorem (see \cite{vonNeumann32}) that the Ces\`aro limit in $h$ of the right hand side of \eqref{eq_continuouscommuting3} is
  $\left(\int_Xf_1d\mu\right)^2=0$ and hence, by Theorem \ref{theorem_vdctrickcontinuous} the proof is complete.

  Assume now that $k>1$ and Theorem \ref{thm_temperedcontconvergence} is true for $k-1$.
  Since each $T_{k,j}$ preserves $\mu$ we can apply $\left(\prod_{j=1}^\ell T_{k,j}^{-q_{k,j}(s)}\right)$ to the right hand side of \eqref{eq_continuouscommuting1}.
  Letting $\tilde f_1=f_1\cdot T_{1,1}^hf_1$ and $\tilde f_i=f_i^2$ for $i=2,\dots,k-1$ we have
  \begin{multline}\label{eq_continuouscommuting2}
     \lim_{\tau \to\infty}\frac1\tau \int_0^\tau \langle a(s+h),a(s)\rangle\dd s=\\
    \lim_{\tau \to\infty}\frac1\tau
  \int_0^\tau \int_X\prod_{i=1}^{k-1}\left(\prod_{j=1}^\ell T_{i,j}^{q_{i,j}(s)}T_{k,j}^{-q_{k,j}(s)}\right)\tilde f_i\cdot f_k^2\dd\mu\dd s
  \end{multline}

It is not hard to check that the right hand side of \eqref{eq_continuouscommuting2} satisfies the conditions of Theorem \ref{thm_temperedcontconvergence} with $k-1$. Therefore
$$\lim_{\tau \to\infty}\frac1\tau \int_0^\tau \langle a(s+h),a(s)\rangle\dd s=\int_Xf_1\cdot
T_{1,1}^hf_1\dd\mu\cdot\prod_{i=2}^k\int_X f_i^2\dd\mu$$
By the mean ergodic theorem we get
$$\lim_{H\to\infty}\frac1H\int_0^H\lim_{\tau \to\infty}\frac1\tau \int_0^\tau \langle a(s+h),a(s)\rangle\dd s\dd
h=\left(\int_Xf_1\dd\mu\right)^2\cdot\prod_{i=2}^k\int_X f_i^2\dd\mu=0$$
The result in question follows now from Theorem \ref{theorem_vdctrickcontinuous}.
\end{proof}
We can now prove Theorem \ref{thm_temperedcontconvergence}:
\begin{proof}[Proof of Theorem \ref{thm_temperedcontconvergence}]

We proceed by induction.
Assume that either $k=1$, or $k>1$ and the theorem has been proved for $k-1$.
After decomposing $f_1=\int_Xf_1\dd\mu+g_1$, where $\int_Xg_1\dd\mu=0$, the left hand side of \eqref{eq_thm_temperedcontconvergence} becomes the sum of two terms.
The first term is
\begin{equation}\label{eq_thm_temperedconvergence2}
\int_Xf_1\dd\mu\cdot\lim_{\tau\to\infty}\frac1\tau\int_0^\tau\prod_{i=2}^k\left[\left(\prod_{j=1}^\ell T_{i,j}^{p_{i,j}(s)}\right)f_i\right]\dd s
\end{equation}
Since we assumed that either $k=1$, or the theorem holds for $k-1$, it follows that \eqref{eq_thm_temperedconvergence2} equals the left hand side of \eqref{eq_thm_temperedcontconvergence} as desired.
The second term in the decomposition is
$$\lim_{\tau\to\infty}\frac1\tau\int_0^\tau\left(\prod_{j=1}^\ell T_{1,j}^{p_{1,j}(s)}\right)g_1\cdot\prod_{i=2}^k\left[\left(\prod_{j=1}^\ell T_{i,j}^{p_{i,j}(s)}\right)f_i\right]\dd s$$
and all that remains to prove is that it equals $0$.
Therefore we have reduced the Theorem to the case when $\int_Xf_1\dd\mu=0$.

Let $\sigma$ be the inverse of $p_{1,1}$.
Since $p_{1,1}$ is an \resfunc, so is $\sigma$.
Let $q_{i,j}=p_{i,j}\circ\sigma$ for each $i=1,\dots,k$ and $j=1,\dots\ell$; observe that these are all \resfunc s.
Since for every $(i,j)\neq(1,1)$, $p_{i,j}(s)/p_{1,1}(s)\to0$ as $s\to\infty$, also $q_{i,j}(s)/s\to0$ as $s\to\infty$.
In view of Lemma \ref{lemma_changeofvariable}, the left hand side of
  \eqref{eq_thm_temperedcontconvergence} can be replaced with
  $$\lim_{\tau\to\infty}\frac1\tau\int_0^\tau \prod_{i=1}^k\left[\left(\prod_{j=1}^\ell T_{i,j}^{q_{i,j}(s)}\right)f_i\right]\dd s$$
The result now follows from Proposition \ref{prop_convergencecontinuous}.
\end{proof}

\section{A multiplicative variant of \DT}\label{sec_katai}

\begin{definition}
  A function $F:\N\to\C$ is \emph{multiplicative} if for any coprime $a,b\in\N$ we have $F(ab)=F(a)F(b)$.

\end{definition}

Classical examples of multiplicative functions include the M\"obius function $\mu(n)$ and the Liouville function $\lambda(n)$.
The Liouville function is defined by $\lambda(p)=-1$ for any prime $p$ and $\lambda(ab)=\lambda(a)\lambda(b)$ for \emph{every} $a,b\in\N$.
The M\"obius function is equal to the Liouville function for every squarefree number, and $\mu(n)=0$ whenever $n$ is divisible by a perfect square $k^2$ with $k\geq2$.
The following theorem, due to K\'atai, has a clear similarity to \DT, in the form of Theorem \ref{thm_vdCinC}.

\begin{theorem}[\cite{Katai86}]\label{thm_katai}
  Let $a:\N\to\C$ be bounded and let $F:\N\to\N$ be a multiplicative function.
  If for any distinct primes $p,q$,
  \begin{equation}\label{eq_katai}\lim_{N\to\infty}\frac1N\sum_{n=1}^Na(pn)\overline{a(qn)}=0
  \end{equation}
  then
  \begin{equation}\label{eq_katai2}
    \lim_{N\to\infty}\frac1N\sum_{k=1}^Na(n)F(n)=0
  \end{equation}
\end{theorem}

Limiting statements of the form \eqref{eq_katai2} are central to analytic number theory.
For instance, if $a(n)\equiv1$ and $F(n)=\mu(n)$ (the M{\"o}bius function), then \eqref{eq_katai2} is equivalent to the Prime Number Theorem. 
If $a(n)$ is replaced with a periodic function, then \eqref{eq_katai2} is equivalent to the prime number theorem in arithmetic progressions.
(Unfortunately, neither of these theorems can be easily derived from Theorem \ref{thm_katai}.)
For $a(n)=e^{2\pi i\alpha n}$, where $\alpha$ is irrational, an application of Theorem \ref{thm_katai} gives a Theorem of Daboussi (cf.\ \cite{Daboussi_Delange82}), stating that, for this choice of $a(n)$, \eqref{eq_katai2} holds for any multiplicative function $F$.
More generally, one can use Theorem \ref{thm_katai} to deduce the following result similar in spirit to (but much simpler than) the main result in \cite{Green_Tao12}:
\begin{theorem}
  Let $f\in\R[x]$ be a polynomial with an irrational coefficient (other than the constant term) and let $F:\N\to\C$ be a bounded multiplicative function.
  Then
  $$\lim_{N\to\infty}\frac1N\sum_{n=1}^NF(n)e^{2\pi if(n)}=0$$
\end{theorem}
\begin{proof}
  In view of Theorem \ref{thm_katai} it suffices to show that \eqref{eq_katai} holds for $a(n)=e^{2\pi i f(n)}$.
  For any primes $p\neq q$, the map $n\mapsto f(pn)-f(qn)$ is a polynomial with an irrational (non-constant) coefficient.
  Therefore
  $$\lim_{N\to\infty}\frac1N\sum_{n=1}^Na(pn)\overline{a(qn)}=\lim_{N\to\infty}\frac1N\sum_{n=1}^N e^{2\pi i \big(f(pn)-f(qn)\big)}=0$$
and this finishes the proof.
\end{proof}

\section{Limits along ultrafilters}\label{sec_ultrafilter}

In this section we will use a variant of \DT{} for limits along idempotent ultrafilters.
Since F\o lner sequences are only available in amenable groups, the form of \DT{} in Theorem \ref{theorem_vdctrickcontinuous} does not apply to non-amenable groups, such as a (non-commutative) free group or $SL(n,\Z)$.
Nevertheless, it is possible to study multiple recurrence in the setting of non-amenable groups through the use of ultrafilters.
The following definitions and facts about ultrafilters can be found, for instance, in \cite{Bergelson10} and \cite{Hindman_Strauss98}.

\begin{definition}
An \emph{ultrafilter} on a countable group $G$ is a family $p$ of subsets of $G$ which is closed under intersections, supersets and which contains exactly one member of any finite partition of $G$.
The set of all ultrafilters is denoted by $\beta G$.
\end{definition}
The set $\beta G$ can be identified with the Stone-\v Cech compactification of $G$.
In particular, the operation on $G$ can be lifted to $\beta G$.
Of special interest are ultrafilters $p\in\beta G$ which satisfy $p\cdot p=p$; such ultrafilters are called \emph{idempotent}.

Given a map $f:G\to K$ from $G$ to a compact Hausdorff space $K$ and an ultrafilter $p\in\beta G$, we denote by $\pl_gf(g)$ the (unique) point $x\in K$ with the property that $\{g\in G:f(g)\in U\}\in p$ for any neighborhood $U$ of $x$.
It follows easily from the definitions that for any $f:G\to K$ from $G$ to a compact Hausdorff space and for any $p\in\beta G$, $\pl_gf(g)$ exists and is unique.
One property which will be important for the exposition in this section is the following:
\begin{equation}\label{eq_ultralimit}
\pl_g\pl_hf(gh)=\pl_gf(g)\qquad\qquad\text{whenever }p\in\beta G\text{ is idempotent}
\end{equation}
The following proposition is the analog of Theorem \ref{theorem_welldistributionperturbation} in the setting of limits along ultrafilters.
It appeared originally as Theorem 2.3 in \cite{Bergelson_McCutcheon07}.
\begin{proposition}\label{prop_plvdc}
  Let $G$ be a countable group, let $H$ be a Hilbert space, let $u:G\to H$ be a bounded sequence and let $p\in\beta G$ be an idempotent ultrafilter.
  Then
  $$\text{If}\qquad\pl_h\pl_g\langle u(hg),u(g)\rangle=0\qquad\text{then}\qquad\pl_g u(g)=0$$
  where the last limit is in the weak topology.
\end{proposition}
The proof presented here is due to Schnell \cite{Schnell07}.
\begin{proof}

For each $N\in\N$ we have that
$$\pl_{g}u(g)=\pl_{g_1}\pl_{g_2}\cdots\pl_{g_N}\frac1N \sum_{k=1}^Nu(g_k\cdots g_N)$$
Taking norms and using the Cauchy-Schwarz inequality, 
we get:
\begin{eqnarray*}
\displaystyle\left\|\pl_{g}u(g)\right\|^2&\leq&\displaystyle \pl_{g_1}\pl_{g_2}\cdots\pl_{g_N}\frac1{N^2} \left\|\sum_{k=1}^Nu(g_k\cdots g_N)\right\|^2\\
&=&\displaystyle \pl_{g_1}\pl_{g_2}\cdots\pl_{g_N}\frac1{N^2} \left\langle\sum_{k=1}^Nu(g_k\cdots g_N),\sum_{l=1}^Nu(g_l\cdots g_N)\right\rangle\\
&=&\displaystyle\frac1{N^2}\sum_{k,l=1}^N \pl_{g_1}\pl_{g_2}\cdots\pl_{g_N}\left
\langle u(g_k\cdots g_N),u(g_l\cdots g_N)\right\rangle
\end{eqnarray*}
Now we use the fact that $p$ is an idempotent ultrafilter to get:

\begin{eqnarray*}
&=&\displaystyle\frac1{N^2}\sum_{k=1}^N\pl_{g_k\to\infty}\|u(n_k)\|^2+ \frac2{N^2}\sum_{k<l}\pl_{g_k\to\infty}\pl_{g_l\to\infty}\langle u(g_kg_l),u(g_l)\rangle\\&=& \displaystyle\frac1N\pl_{g\to\infty}\|u(g)\|^2
\end{eqnarray*}
Since $N$ can be chosen arbitrarily we conclude that $\displaystyle\pl_{n\to\infty}x_n=0$.
\end{proof}

Proposition \ref{prop_plvdc} can be used to give a quick proof of the following polynomial extension of Khintchine's recurrence theorem (see, for instance, \cite[Theorem 5.1]{Bergelson96}).
 \begin{theorem}[cf. {\cite[Theorem 3.12]{Bergelson96}}]\label{theorem_polkhintchine}
   Let $(X,\mu,T)$ be an invertible measure preserving system, let $q\in\Z[x]$ be a polynomial satisfying $q(0)=0$ and let $p\in\beta\Z$ be an idempotent ultrafilter.
   Then there exists an orthogonal projection $P:L^2(X)\to L^2(X)$ such that for any $f\in L^2(X)$
   $$\pl_nf\circ T^{q(n)}=Pf\qquad\qquad\text{in the weak topology}$$
 In particular, for any $\epsilon>0$, the set $R:=\big\{n\in\Z:\mu(A\cap T^{q(n)}A)>\mu(A)^2-\epsilon\big\}$ is in $p$ and hence non-empty.
 \end{theorem}

In fact, Theorem \ref{theorem_polkhintchine} implies that $R$ is an IP$^*$-set (see Definition \ref{def_ipsets} below).
This result was first obtained in \cite{Bergelson_Furstenberg_McCutcheon96} (see also section 3 in \cite{Bergelson96}).
One can deduce combinatorial corollaries from recurrence results using Furstenberg's correspondence principle (see \cite{Furstenberg77,Furstenberg81,Bergelson87b}).
In particular, Theorem \ref{theorem_polkhintchine} implies that for any set $A\subset\N$ with positive upper density $\bar d(A):=\limsup|A\cap[1,N]|/N>0$ and for any $q\in\Z[x]$ with $q(0)=0$, there exist many $x,y\in\N$ such that $\{x,x+q(y)\}\subset A$.


Next we define the notion of a $p$-mixing measure preserving system
.

\begin{definition}
  Let $G$ be a countable group and let $p\in\beta G$.
  A measure preserving system $(X,\mu,(T_g)_{g\in G})$ is called \emph{$p$-mixing} if for every $f\in L^2(X)$ we have $\pl_g T_gf=\int_Xf\dd\mu$ in the weak topology.

  When $G=\Z$ and $p\in\beta\Z$, a measure preserving system $(X,\mu,T)$ is called \emph{totally $p$-mixing} if for every $\ell\in\N$ the system $(X,\mu,T^\ell)$ is $p$-mixing.
  Equivalently, $(X,\mu,T)$ is totally $p$-mixing if for every $f\in L^2(X)$ and every $\ell\in\N$ we have $\pl_n T^{\ell n}f=\int_Xf\dd\mu$ in the weak topology.
\end{definition}

\subsection{A polynomial ergodic theorem for totally $p$-mixing systems}\label{sec_mildmixingpet}

In this subsection we will use Proposition \ref{prop_plvdc} to deduce the following result.


\begin{theorem}\label{thm_pmixingPET}
  Let $p\in\beta\N$ be an idempotent ultrafilter and let $(X,\mu,T)$ be a totally $p$-mixing measure preserving system.
  Let $p_1,\dots,p_k\in\Z[x]$ be polynomials such that for every $i\neq j$, the polynomial $p_i-p_j$ is not constant.
  For every $f_1,\dots,f_k\in L^\infty(X)$ we have:
  \begin{equation}\label{eq_thm_pmixingPET}\pl_n \prod_{i=1}^kT^{p_i(n)}f_i=\prod_{i=1}^k\int_Xf_i\dd\mu\qquad\text{in the weak topology}\end{equation}
\end{theorem}
If the system is $p$-mixing but not totally $p$-mixing, the conclusion of Theorem \ref{thm_pmixingPET} may not hold.
In fact, there are $p$-mixing systems satisfying $\pl_n T^{2n}f=f$ for every $f\in L^2$, see \cite{Bergelson_Kasjan_Lemanczyk14}.

The proof of Theorem \ref{thm_pmixingPET} utilizes the PET-induction scheme developed in \cite{Bergelson87}.
Two polynomials $p_1,p_2\in\Z[x]$ are \emph{equivalent} if they have the same degree and leading coefficient, in other words, if $p_2-p_1$ has degree strictly smaller than $p_2$.
Given a finite family $P=\{p_1,\dots,p_k\}$ of polynomials, let $d$ be the maximal degree of the $p_i$ and, for each $j=1,\dots,d$ let $s_j$ be the number of equivalence classes in $P$ of polynomials of degree $j$.
The vector $(s_1,\dots,s_d)$ is called the \emph{characteristic vector} of $P$.
Since $s_d>0$, the characteristic vector is unique.
For example, the family $P=\{x,2x-1,3x,x^3+2x^2,x^3+1\}$ has characteristic vector $(3,0,1)$.

We order the characteristic vectors by letting $(s_1,\dots,s_d)<(\tilde s_1,\dots,\tilde s_{\tilde d})$ if and only if either $d<\tilde d$ or both $d=\tilde d$ and the maximum $j$ for which $s_j\neq\tilde s_j$ satisfies $s_j<\tilde s_j$.
For example $(1,2,3)<(0,0,0,1)$ and $(9,3,5,2,4)<(1,7,6,2,4)$.

We can now prove Theorem \ref{thm_pmixingPET}.
\begin{proof}[Proof of Theorem \ref{thm_pmixingPET}]

We will prove this theorem by induction on the characteristic vector of $P=\{p_1,\dots,p_k\}$.
If $P$ has characteristic vector $(1)$, then $P$ must consist of a single linear polynomial and \eqref{eq_thm_pmixingPET} holds trivially.
Next assume that the characteristic vector $(s_1,\dots,s_d)$ of $P$ is larger than the vector $(1)$ and that Theorem \ref{thm_pmixingPET} has been proved for any family $\tilde P$ with a smaller characteristic vector.
Without loss of generality we can assume that $\int_Xf_i=0$ for every $i=1,\dots,k$.
We will also assume that the polynomials in $P$ have $0$ constant term.

Let $q\in P$ have the smallest degree.
If we denote by $e$ the degree of $q$, then $s_e\geq1$ and $s_j=0$ for each $j<e$.

Let $u_n=\prod_{p\in P}T^{p(n)}f_p$.
We need to show that $\pl_n u_n=0$.
We consider separately two cases.
\begin{itemize}
\item[Case 1] ($e>1$).
Let $Q\subset P$ be the equivalence class of $q$.
With Proposition \ref{prop_plvdc} in mind, we look at the inner product $\langle u_{n+h},u_n\rangle$:
\begin{eqnarray*}\langle u_{n+h},u_n\rangle&=&\int_X\prod_{p\in P}T^{p(n+h)}f_p\cdot T^{p(n)}f_p\dd\mu\\&=& \int_X\prod_{p\in P}T^{p(n+h)-q(n)}f_p\cdot T^{p(n)-q(n)}f_p\dd\mu\\&=&\int_Xf_q\cdot\prod_{\tilde p\in P_h}T^{\tilde p(n)}\tilde f_{\tilde p}\dd\mu\end{eqnarray*}
where $P_h$ is the family of polynomials consisting of the maps $\tilde p:n\mapsto p(n)-q(n)$ for $p\in P\setminus\{q\}$, associated with the function $\tilde f_{\tilde p}=f_p$, and the maps $\tilde p_h:n\mapsto p(n+h)-q(n)$ for any $p\in P$, associated with the function $\tilde f_{\tilde p}=f_p$.

Notice that, for every $p\in P\setminus Q$ and every $h\in\N$, the polynomial $n\mapsto p(n)-q(n)$ is equivalent to $n\mapsto p(n+h)-q(n)$ and both have the same degree as $p$.
Next, observe that whenever $p_1,p_2\in P\setminus Q$, the polynomial $n\mapsto p_1(n)-q(n)$ is equivalent to $n\mapsto p_2(n)-q(n)$ if and only if $p_1$ is equivalent to $p_2$.
Finally note that if $p\in Q$, then the degree of $p(n)-q(n)$ and $p(n+h)-q(n)$ is strictly smaller than $e$.
Since all polynomials in $P$ have degree at least $2$, the family $P_h$  contains no two polynomials $p_1,p_2$ whose difference $p_1-p_2$ is a constant (for all but finitely many $h$).
It follows from the observations above that the characteristic vector of $P_h$ is strictly smaller than that of $P$.
From the induction hypothesis we obtain
$$\pl_n\langle u_{n+h},u_n\rangle=\int_Xf_q\left(\pl_n\prod_{\tilde p\in P_h}T^{\tilde p(n)}\tilde f_{\tilde p}\right)\dd\mu=0$$
and hence, the desired result $\pl_n u_n=0$ follows from the ultrafilter version of \DT{} (Proposition \ref{prop_plvdc}).

\item[Case 2] ($e=1$).
In this case, for any $p\in P$ with degree $1$, the polynomials $p(n)-q(n)$ and $p(n+h)-q(n)$ have a constant difference.
To overcome this difficulty, let $n\mapsto a_in$, $i=1,\dots,s$ be all the polynomials of degree $1$ in $P$, and assume that $q(n)=a_1n$.
Observe that all the $a_i\in\Z$ are distinct.
Let $Q$ be the set of polynomials in $P$ with degree at least $2$.
We now have
\begin{eqnarray*}\langle u_{n+h},u_n\rangle&=&\int_X\prod_{i=1}^sT^{a_i(n+h)}f_i\cdot T^{a_in}f_i\cdot\prod_{p\in Q}T^{p(n+h)}f_p\cdot T^{p(n)}f_p\dd\mu\\&=&\int_X\prod_{i=1}^sT^{(a_i-a_1)n}\left(T^{a_ih}f_i\cdot f_i\right)\cdot\prod_{p\in Q}T^{p(n+h)-q(n)}f_p\cdot T^{p(n)-q(n)}f_p\dd\mu\\&=&\int_XT^{a_1h}f_1\cdot f_1\cdot\prod_{\tilde p\in P_h}T^{\tilde p(n)}\tilde f_{\tilde p}\dd\mu\end{eqnarray*}
where  the family $P_h$ consists of the polynomials $\tilde p:n\mapsto(a_i-a_1)n$ for $i=2,\dots,s$, associated with the function $\tilde f_{\tilde p}=T^{a_i}f_i\cdot f_i$ together with the polynomials $\tilde p:n\mapsto p(n)-q(n)$ and $\tilde p_h:n\mapsto p(n+h)-q(n)$, both associated with the function $\tilde f_{\tilde p}=\tilde f_{\tilde p_h}=f_p$.

Applying the same reasoning as in the first case we deduce that the family $P_h$ has a smaller characteristic vector than $P$.
It follows from the induction hypothesis that
$$\pl_n\langle u_{n+h},u_n\rangle=\int_XT^{a_1h}f_1\cdot f_1\dd\mu\cdot\prod_{p\in P_h}\int_X\tilde f_p\dd\mu$$
Since each $\tilde f_p\in L^\infty(X)$, the absolute value of the second term in this expression is bounded, say by $M$.
Hence we deduce that
$$\pl_h\left|\pl_n\langle u_{n+h},u_n\rangle\right|\leq\pl_h\left|\int_XT^{a_1h}f_1\cdot f_1\dd\mu\right|\cdot M=0$$
Finally, applying Proposition \ref{prop_plvdc} we obtain the desired conclusion.\qedhere
\end{itemize}
\end{proof}


Recall that a measure preserving system is called \emph{mildly mixing} if for any $f\in L^2(X)$ and any $p=p+p\in\beta\N$ we have $\pl_n T^nf=\int_Xf\dd\mu$.
We note in passing that every mixing system is mildly mixing and every mildly mixing system is weakly mixing, but both inclusions are strict (see, for example \cite{Furstenberg_Weiss77}).

Since for every idempotent $p\in\beta\N$ and $\ell\in\N$, the ultrafilter $\ell p\in\beta\N$ \footnote{The ultrafilter $\ell p$ is defined by $A\in \ell p\iff A/\ell\in p$ (where, in turn, $A/\ell$ is defined by $x\in A/\ell\iff \ell x\in A$).} is also idempotent, it follows that a mildly mixing system is in fact totaly $p$-mixing for every idempotent ultrafilter $p$.
This allows us to deduce from Theorem \ref{thm_pmixingPET}  the following corollary:
\begin{corollary}\label{cor_mmpet}
  Let $(X,\mu,T)$ be a mildly mixing measure preserving system and let $p_1,\dots,p_k\in\Z[x]$ be polynomials.
  If for every $i\neq j$, the polynomial $p_i-p_j$ is not constant, then for every $f_1,\dots,f_k\in L^\infty(X)$ we have \eqref{eq_thm_pmixingPET}.
  In particular, for every $\epsilon>0$ and $f_0\in L^2(X)$, the set
  $$\left\{n:\left|\int_Xf_0\cdot\prod_{i=1}^kT^{p_i(n)}f_i\dd\mu-\prod_{i=0}^k\int_Xf_i\dd\mu\right|<\epsilon\right\}$$
  is IP$^*$.
\end{corollary}

\subsection{An ultrafilter analog of joint ergodicity}\label{sec_jointpergodicity}
The mean ergodic theorem states that a measure preserving system $(X,{\mathcal B},\mu,T)$ is ergodic if and only if for every $f\in L^2(X)$, the averages $\tfrac1N\sum_{n=1}^NT^nf$ converge in norm to $\int f\dd\mu$ as $N\to\infty$.
The following definition is an extension of the notion of ergodicity to several tranformations:
\begin{definition}Let $k\in\N$ and let $T_1,\dots,T_k$ be invertible measure preserving transformations on the same probability space $(X,{\mathcal B},\mu)$.
We say that the maps $T_1,\dots,T_k$ are \emph{jointly ergodic} if for any $f_1,\dots,f_k\in L^\infty(X)$
$$\lim_{N\to\infty}\frac1N\sum_{n=1}^NT_1^nf_1\cdots T^n_kf_k=\int_Xf_1\dd\mu\cdots\int_Xf_k\dd\mu\qquad\text{in }L^2$$
\end{definition}
The notion of joint ergodicity was introduced in \cite{Berend_Bergelson84}, where the authors gave the following necessary and sufficient conditions for a tuple of invertible commuting transformations to be jointly ergodic.

\begin{theorem}[{\cite{Berend_Bergelson84}}]\label{thm_jointoriginal}
  Let $k\in\N$ and let $T_1,\dots,T_k$ be invertible measure preserving transformations on the same probability space $(X,{\mathcal B},\mu)$.
  Then $T_1,\dots,T_k$ are jointly ergodic if and only if all the transformations $T_iT_j^{-1}$ and $T_1\times\cdots\times T_k$ are ergodic.
\end{theorem}
Later, in \cite{Bergelson_Rosenblatt88}, the definition of joint ergodicity was (somewhat modified and) extended to actions of more general groups.
Let $G$ be a countable group\footnote{In \cite{Bergelson_Rosenblatt88} the setup is that of general locally compact amenable groups.} and let $(T_1^g)_{g\in G},\dots,(T^g_k)_{g\in G}$ be measure preserving actions of $G$ on a probability space $(X,{\mathcal B},\mu)$.
We say that these actions \emph{commute} if for any $g,h\in G$ and $i\neq j$ we have $T_i^gT_j^h=T_j^hT_i^g$.
Note that when $G$ is non-commutative, this equality may fail for $i=j$.

In the above setting, whenever $1\leq i\leq j\leq k$, we denote by $T_{[i,j]}^g$ the measure preserving map $T_{[i,j]}^g=T_i^gT_{i+1}^g\cdots T^g_j$.
Observe that the commutativity assumption implies that $\big(T_{[i,j]}^g\big)_{g\in G}$ is an action of $G$.
\begin{definition}
  Let $G$ be a countable amenable group, let $(F_N)_{N\in\N}$ be a F\o lner sequence in $G$, let $k\in\N$ and let $(T_1^g)_{g\in G},\dots,(T^g_k)_{g\in G}$ be commuting measure preserving actions of $G$ on a probability space $(X,{\mathcal B},\mu)$.
  The actions $T_1,\dots, T_k$ are called \emph{jointly ergodic} if for every $f_1,\dots,f_k\in L^\infty(X)$
   $$\lim_{N\to\infty}\frac1{|F_N|}\sum_{g\in F_N}\prod_{i=1}^kT_{[1,i]}^gf_i=\prod_{i=1}^k\int_Xf_i\dd\mu\qquad\text{in }L^2(X)$$
\end{definition}
One has the following theorem.
\begin{theorem}[cf. {\cite[Theorems 2.4 and 2.6]{Bergelson_Rosenblatt88}}]\label{thm_joinergamenable}
  Let $G$ be a countable amenable group, let $(F_N)_{N\in\N}$ be a F\o lner sequence in $G$, let $k\in\N$ and let $(T_1^g)_{g\in G},\dots,(T^g_k)_{g\in G}$ be commuting measure preserving actions of $G$ on a probability space $(X,{\mathcal B},\mu)$.
  Then the actions $T_1,\dots,T_k$ are jointly ergodic if and only if
  $$  \big(T_{[i,j]}^g\big)_{g\in G}\qquad\text{ are ergodic for every }i\leq j$$
  and $T_1\times T_{[1,2]}\times\cdots\times T_{[1,k]}$ is ergodic on $(X^k,{\mathcal B}^{\otimes k},\mu^{\otimes k})$.
\end{theorem}
One can show that when $G=\Z$, Theorem \ref{thm_joinergamenable} implies Theorem \ref{thm_jointoriginal}.

In this subsection we establish necessary and sufficient conditions, analogous to those obtained in \cite{Berend_Bergelson84} and \cite{Bergelson_Rosenblatt88}, for joint $p$-mixing.

\begin{definition}
  Let $G$ be a countable group, let $p\in\beta G$ be an idempotent ultrafilter, let $k\in\N$ and let $(T_1^g)_{g\in G},\dots,(T^g_k)_{g\in G}$ be commuting measure preserving actions of $G$ on a probability space $(X,{\mathcal B},\mu)$.
  The actions $T_1,\dots, T_k$ are called \emph{jointly $p$-mixing} if for every $f_1,\dots,f_k\in L^\infty(X)$
   \begin{equation}\label{eq_thm_pJointlyMixing}
  \pl_g\prod_{i=1}^kT_{[1,i]}^gf_i=\prod_{i=1}^k\int_Xf_i\dd\mu\qquad\text{weakly in }L^2(X)
  \end{equation}
  \end{definition}

\begin{theorem}\label{thm_pJointlyMixing}
  Let $G$ be a countable group, let $p\in\beta G$ be an idempotent ultrafilter, let $k\in\N$ and let $(T_1^g)_{g\in G},\dots,(T^g_k)_{g\in G}$ be commuting measure preserving actions of $G$ on a probability space $(X,{\mathcal B},\mu)$.
  Then the actions $T_1,\dots,T_k$ are jointly $p$-mixing if and only if
  \begin{equation}\label{eq_thm_pJointlyMixing2}
  \big(T_{[i,j]}^g\big)_{g\in G}\qquad\text{ are }p\text{-mixing for every }i\leq j
  \end{equation}
\end{theorem}

\begin{proof}
We first prove that \eqref{eq_thm_pJointlyMixing2} implies joint $p$-mixing, i.e., \eqref{eq_thm_pJointlyMixing}.
  We proceed by induction on $k\in\N$.
  When $k=1$, \eqref{eq_thm_pJointlyMixing} reduces to the very definition of $p$-mixing.

  Assume now that the result has been proved for $k-1$ transformations.
  Since both sides of \eqref{eq_thm_pJointlyMixing} are linear on each $f_i$, and for constant $f_k$ the result follows by induction, we can reduce the statement to the case when $\int_Xf_k\dd\mu=0$.
  Now the right hand side of \eqref{eq_thm_pJointlyMixing} vanishes, so all we need to show is that $\pl_g u_g=0$, where $u_g=\prod_{i=1}^kT_{[1,i]}^gf_i$.
  In order to obtain this we will resort to \DT, in the form of Proposition \ref{prop_plvdc}.
  Let $h\in G$ be an arbitrary element different from the identity.
  \begin{eqnarray*}\langle u_{gh},u_g\rangle&=&\int_X\prod_{i=1}^kT_{[1,i]}^{gh}f_i\cdot T_{[1,i]}^gf_i\dd\mu\\ &=&\int_X\prod_{i=1}^kT_{[1,i]}^g\Big(T_{[1,i]}^hf_i\cdot f_i\Big)\dd\mu \\&=&\int_XT_1^g\left(T_1^hf_1\cdot f_1\cdot\prod_{i=2}^kT_{[2,i]}^g\Big(T_{[1,i]}^hf_i\cdot f_i\Big)\right)\dd\mu \\&=&\int_XT_1^hf_1\cdot f_1\cdot\prod_{i=2}^kT_{[2,i]}^g\Big(T_{[1,i]}^hf_i\cdot f_i\Big)\dd\mu\end{eqnarray*}
  Putting $\tilde T_i=T_{i+1}$ and $\tilde f_i=T_{[1,i+1]}^hf_{i+1}\cdot f_{i+1}$ for each $i=1,\dots,k-1$, observe that $\tilde T_{[i,j]}=T_{[i+1,j+1]}$.
  In particular, by induction we obtain that
  $$\pl_g\prod_{i=2}^kT_{[2,i]}^g\Big(T_{[1,i]}^hf_i\cdot f_i\Big)=\pl_g\prod_{i=1}^{k-1}\tilde T_{[1,i]}^g\tilde f_i=\prod_{i=1}^{k-1}\int_X\tilde f_i\dd\mu$$
  and hence we deduce that
  $$\pl_g\langle u_{gh},u_g\rangle=\int_XT_1^hf_1\cdot f_1\cdot\prod_{i=1}^{k-1}\int_X\tilde f_i\dd\mu=\prod_{i=1}^k\int_XT_{[1,i]}^hf_i\cdot f_i\dd\mu$$
  Finally, taking the $\pl$ in $h$ we conclude that
  $$\pl_h\pl_g\langle u_{gh},u_g\rangle=\prod_{i=1}^k\left(\int_Xf_i\dd\mu\right)^2=0$$
  and the result follows from Proposition \ref{prop_plvdc}.

Now we prove that joint $p$-mixing implies \eqref{eq_thm_pJointlyMixing2}.
It follows directly from \eqref{eq_thm_pJointlyMixing}, putting all functions but one equal to the constant $1$, that for any $i$ one has $\pl_g T_{[1,i]}^gf=\int_Xfd\mu$.
In particular $T_{[1,i]}$ is $p$-mixing for every $i$.
Next, let $1<i\leq j\leq k$ and let $f\in L^\infty(X)$. 
We can assume, without loss of generality, that $\int_Xf\dd\mu=0$, and we need to show that $\tilde f:=\pl_g T_{[i,j]}^gf$ is equal to $0$ (where the limit is in the weak topology of $L^2(X)$).

Take $f_j=f$, $f_{i-1}=\tilde f$ and $f_r=1$ for all the other $r$.
Then from \eqref{eq_thm_pJointlyMixing} we obtain
$$0=\int_Xf\dd\mu\int_X\tilde f\dd\mu=\pl_g\prod_{r=1}^kT^g_{[1,r]}f_r=\pl_gT^g_{[1,i-1]}\big(\tilde f\cdot T_{[i,j]}^gf\big)=\int_X\big(\tilde f\big)^2\dd\mu$$
We conclude that indeed $\tilde f=0$.
Observe that in order to invoke \eqref{eq_thm_pJointlyMixing} we used implicitly the (easily checkable) fact that $\tilde f\in L^\infty$.
\end{proof}

\section{A topological variant of the Difference Theorem}\label{sec_topologicalvdC}

As we have seen in previous sections, several versions of \DT{} are useful in the study of multiple recurrence in ergodic theory.
The phenomenon of (multiple) recurrence can also be studied from a topological standpoint.
In this setting there is no good analog of ergodic averages, and hence the Hilbertian \DT{} (Theorem \ref{theorem_vdctrickcontinuous}) does not
apply.
Nevertheless, one can still adapt the fundamental idea of complexity reduction present in \DT{} to this setup.
This method was developed in \cite{Bergelson_Leibman96}, with a somewhat related ideas appearing earlier in \cite{Blaszczyk_Plewik_Turek89}.

\begin{definition}\label{def_ipsets}
  Let ${\mathcal F}$ denote the family of all finite non-empty subsets of $\N$.
  Given an increasing sequence $(n_k)_{k\in\N}$ we define the \emph{IP-set generated by $(n_k)$} to be the (image of the) map $n:{\mathcal F}\to\N$ defined by $n_\alpha=\sum_{k\in\alpha}n_k$.
  A set $A\subset\N$ is called an \emph{IP$^*$ set} if it has nonempty intersection with every IP-set.
\end{definition}
A famous theorem of Hindman \cite{Hindman74} states that for any finite partition of an IP-set, one of the cells of the partition contains an IP-set.
Equivalently, it states that a finite intersection of IP$^*$ sets is still an IP$^*$ set.

In this section we will deal with topological dynamical systems, namely with pairs $(X,T)$ where $X$ is a compact Hausdorff space and $T:X\to X$ is a homeomorphism.
The system $(X,T)$ (or the map $T$) is called \emph{minimal} if there is no non-empty compact subset $Y\subset X$ such that $T^{-1}Y\subset Y$.
\begin{definition}
    A family of sequences $A=\{a_1(n),\dots,a_k(n)\}$ of integers is a \emph{family of multiple IP topological recurrence} if for any minimal system $(X,T)$, any non-empty open set $U\subset X$ and any IP set
    $(n_\alpha)_{\alpha\in{\mathcal F}}$, there exists $\alpha\in{\mathcal F}$ such that
  $$U\cap T^{-a_1(n_\alpha)}U\cap\cdots\cap T^{-a_k(n_\alpha)}U\neq\emptyset$$
  Equivalently, $A$ is a family of multiple IP topological recurrence if the set
  $$\{n\in\Z:U\cap T^{-a_1(n)}U\cap\cdots\cap T^{-a_k(n)}U\neq\emptyset\}$$
  is an IP$^*$ set.
\end{definition}

\begin{remark}\label{remark_iprechindman}
  It is an easy consequence of Hindman's theorem that whenever $\{a_1(n),\dots,a_k(n)\}$ and
$\{b_1(n),\dots,b_\ell(n)\}$ are families of multiple IP topological recurrence,
$(X,T)$, $(Y,S)$ are minimal systems and $U\subset X$, $V\subset Y$ are open
sets, there exists $n\in\N$ such that
$$U\cap T^{-a_1(n)}U\cap\cdots\cap T^{-a_k(n)}U\neq\emptyset\neq V\cap S^{b_1(n)}V\cap\cdots\cap S^{-b_\ell(n)}V$$
In fact, the set of such $n$ is IP$^*$.
\end{remark}

Multiple IP topological recurrence has strong connections with combinatorics.
This connection between topological dynamics and Ramsey theory was established by Furstenberg and Weiss in \cite{Furstenberg_Weiss78}.
Using the methods developed in \cite{Furstenberg_Weiss78} one can show that if $\{a_1(n),\dots,a_k(n)\}$ is a family of multiple IP topological recurrence, then for any finite partition of the natural numbers $\N=C_1\cup\cdots\cup C_r$ there exists $i\in\{1,\dots,r\}$, $x\in C_i$ and $n\in\N$ such that $x+a_j(n)\in C_i$ for every $j=1,\dots,k$.
Moreover, $n$ can be chosen from any prescribed IP-set.

\begin{definition}\label{def_syndthick}
  A subset $S\subset\N$ is \emph{syndetic} if it has bounded gaps\footnote{Note that a set $S\subset\N$ is syndetic if and only if it has positive lower Banach density (cf. Definition \ref{def_besicovitch}).}; in other words if for some $n\in\N$, the union $S\cup(S-1)\cup\cdots\cup(S-n)=\N$.
A set $T\subset\N$ is \emph{thick} if it contains arbitrarily long intervals.
A set $A\subset\N$ is called \emph{piecewise syndetic} if it is the intersection $A=S\cap T$ of a syndetic set $S$ and a thick set $T$.
\end{definition}

For any finite partition of a piecewise syndetic set, one of the cells is piecewise syndetic.
More generally, for any semigroup $G$, a set $S\subset G$ is called (left) \emph{syndetic} if there exists a finite set $F\subset G$ such that $F^{-1}S:=\{g\in G:\exists x\in F\text{ s.t. }xg\in S\}=G$.
A set $T\subset G$ is called (right) \emph{thick} if it has non-empty intersection with every syndetic set and a set $A\subset G$ is called piecewise syndetic if it equals the intersection of a syndetic set and a thick set.

One can show that if $\{a_1(n),\dots,a_k(n)\}$ is a family of multiple IP topological recurrence and $A\subset\N$ is a piecewise syndetic set, then there exists $x\in A$ and $n\in\N$ such that $\big\{x,x+a_1(n),\dots,x+a_k(n)\big\}\subset A$.
Moreover, $n$ can be chosen from any prescribed IP-set.

  The family $\{a(n)\}$ consisting the single degenerated sequence $a(n)\equiv0$ is trivially a family of multiple IP topological recurrence.
  The following key lemma, which is the analog of \DT{} in this setting, will allow us to quickly obtain some interesting applications:

\begin{lemma}[Reduction of complexity (IP version)]\label{lemma_topvdcIP}
  Let $\{a_1(n),\dots,a_s(n)\}$, $n\in\N$ be a finite family of sequences of positive integers and put $a_\ell(0)=0$ for all
  $\ell=1,\dots,s$.
  If for any finite set $F\subset\N\cup\{0\}$ the family of sequences
  $$\big\{n\mapsto a_\ell(n+h)-a_1(n)-a_\ell(h)\mid h\in F,~\ell\in\{1,\dots,s\}\big\}$$ is a family of multiple IP topological recurrence, then so is the family $\{a_1(n),\dots,a_k(n)\}$, $n\in\N$.
\end{lemma}

\begin{proof}

  Let $X$ be a compact Hausdorff space, let $T:X\to X$ be a minimal homeomorphism and let $U\subset X$ be a non-empty
  open set.
  Since the system is minimal, there exists some $r\in\N$ such that $U\cup T^{-1}U\cup T^{-2}U\cup\cdots\cup
  T^{-r}U=X$.
  Let $(n_\alpha)_{\alpha\in{\mathcal F}}$ be an arbitrary IP-set in $\N$.
  Let $\alpha_1\in{\mathcal F}$ be arbitrary, let $U_1=U$ and $t_1=0$.
  For each $k\in\N$ we will construct
  inductively $\alpha_k\in{\mathcal F}$, $t_k\in\{0,\dots,r\}$ and $U_k\subset T^{-t_k}U$ a non-empty open set such
  that
  \begin{equation}\label{eq_topvdc_inductionhypothesis3}T^{-a_\ell(n_{\alpha_{j+1}\cup\alpha_{j+2}\cup\cdots\cup\alpha_k})}U_j\supset
  U_k\qquad\text{ for each }j<k\text{ and }\ell\in\{1,\dots,s\}\end{equation}

  For simplicity, whenever $j\leq k$, denote $n_{\alpha_{j+1}\cup\alpha_{j+2}\cup\cdots\cup\alpha_k}$ by $m_{j,k}$, with the understanding that $m_{k,k}=0$.
  Assume we have already chosen $\alpha_i,t_i,U_i$ for $i<k$.
  Consider the family of sequences of the form $n\mapsto a_\ell(n+m_{j,k-1})-a_1(n)-a_\ell(m_{j,k-1})$ for each
  $j\in\{1,\dots,k-2\}$ and $\ell\in\{1,\dots,s\}$.
  By hypothesis this is a family of multiple IP topological recurrence, so there exists some $\alpha_k\in{\mathcal F}$, disjoint
  from $\bigcup_{i<k}\alpha_i$ and such that
  \begin{equation}\label{eq_topvdc_inductionconclusion3}V_k:=\bigcap_{\ell=1}^s\bigcap_{j=1}^{k-1}
  T^{-\big(a_\ell(n_{\alpha_k}+m_{j,k-1})-a_1(n_{\alpha_k})-a_\ell(m_{j,k-1})\big)}U_{k-1}\neq\emptyset\end{equation}
  Next choose $t_k\in\{0,\dots,r\}$ such that $U_k:=T^{-a_0(n_{\alpha_k})}V_k\cap T^{-t_k}U\neq\emptyset$.
  To see how \eqref{eq_topvdc_inductionhypothesis3} follows from \eqref{eq_topvdc_inductionconclusion3}, let $x\in
  U_k$, let $\ell\in\{1,\dots,s\}$ and let $j<k$.
  If $j=k-1$ then the first inclusion below degenerates to a trivial equality.
  In any case we have:
  \begin{eqnarray*}
    T^{-a_\ell(m_{j,k})}U_j&=&
    T^{-a_1(n_{\alpha_k})}T^{-\big(a_\ell(m_{j,k})-a_1(n_{\alpha_k})-a_\ell(m_{j,k-1})\big)}T^{-a_\ell(m_{j,k-1})}U_j\\
    &\supset& T^{-a_1(n_{\alpha_k})}T^{-\big(a_\ell(m_{j,k})-a_1(n_{\alpha_k})-a_\ell(m_{j,k-1})\big)}U_{k-1}\\
    &\supset& T^{-a_1(n_{\alpha_k})}V_k\\ &\supset& U_k
  \end{eqnarray*}
  Finally, take $k>j$ such that $t_k=t_j=:t$ (such a pair must exist since each $t_i$ belongs to the finite set
  $\{0,\dots,r\}$).
  From \eqref{eq_topvdc_inductionhypothesis3} it follows that $T^{-a_\ell(m_{j,k})}U_j\supset U_k$ for each
  $\ell\in\{1,\dots,s\}$, which implies that $U_k\subset(T^{-t}U)\cap T^{-a_\ell(m_{j,k})}(T^{-t}U)=T^{-t}(U\cap
  T^{-a_\ell(m_{j,k})}U)$ and therefore
  \[\emptyset\neq T^t(U_k)\subset U\cap\bigcap_{\ell=1}^sT^{-a_\ell(m_{j,k})}U\qedhere\]
\end{proof}

\begin{remark}
  One can state and prove the previous lemma (using the same method) for the case of regular (i.e., not IP) recurrence, weakening both the assumptions and the conclusion.
\end{remark}

We now present some applications of Lemma \ref{lemma_topvdcIP}.
\begin{theorem}[Topological IP van der Waerden Theorem (cf. {\cite[Theorem 3.2]{Furstenberg_Weiss78}})]\label{theorem_topvdw}
  Let $(X,T)$ be a minimal system and let $U\subset X$ be a non-empty open set.
  Then for any $k\in\N$ and any IP-set $(n_\alpha)_{\alpha\in{\mathcal F}}$ there exists $\alpha\in{\mathcal F}$ such
  that
  $$U\cap T^{-n_\alpha}U\cap T^{-2n_\alpha}\cap\cdots\cap T^{-kn_\alpha}U\neq\emptyset$$
\end{theorem}
\begin{proof}
  We prove the claim by induction on $k$.
  The case $k=0$ is vacuously true.
  Next assume the result is true for $k-1$.
  We want to show that the family of sequences $\{n,2n,\dots,kn\}$ is a family of multiple IP topological
  recurrence.
  Letting $a_\ell(n)=\ell n$ we see that $a_\ell(n+h)-a_1(n)-a_\ell(h)=(\ell-1)n$ for any $h\in\N\cup\{0\}$.
  Therefore for any finite subset $F\subset\N\cup\{0\}$ we have
  $$\big\{n\mapsto a_\ell(n+h)-a_1(n)-a_\ell(h)\mid h\in
  F,~\ell\in\{1,\dots,s\}\big\}\subset\{0,n,2n,\dots,(k-1)\ell\}$$
By the induction hypothesis, this is a family of multiple IP topological recurrence.
It follows from Lemma \ref{lemma_topvdcIP} that $\{n,2n,\dots,kn\}$ is also a family of multiple IP topological
recurrence, which finishes the proof.
\end{proof}

As a corollary we obtain van der Waerden's celebrated theorem \cite{vdWaerden27} stating that for any finite partition $\N=C_1\cup\cdots\cup C_r$, one of the cells $C_i$ contains arbitrarily long arithmetic progressions (Theorem \ref{theorem_topvdw} actually implies that any piecewise syndetic set contains arbitrarily long arithmetic progressions).
Moreover, the common difference can be chosen from any prescribed IP-set.

\begin{theorem}[Topological IP S\'ark\"ozy theorem (cf. Subsection 1.3 in {\cite{Bergelson_Leibman96}})]\label{thm_sarkozy}
Let $(X,T)$ be a minimal system and let $U\subset X$ be a non-empty open set.
  Then for any IP-set $(n_\alpha)_{\alpha\in{\mathcal F}}$ there exists $\alpha\in{\mathcal F}$ such that
  $$U\cap T^{-n_\alpha^2}U\neq\emptyset$$
\end{theorem}

\begin{proof}
  We need to show that the family $\{n^2\}$ consisting of a single sequence $a_1(n)=n^2$, $n\in\N$ is a family of IP
  topological recurrence; we will use Lemma \ref{lemma_topvdcIP} to that end.
  Let $F\subset\N\cup\{0\}$ be an arbitrary finite set and let $k=2\max F$.
  One can easily check that
  $$\big\{n\mapsto a_1(n+h)-a_1(n)-a_1(h)\mid h\in F\big\}\subset\{0,n,2n,\dots,kn\}$$
  It follows from Theorem \ref{theorem_topvdw} that this is a family of multiple IP topological recurrence.
  Applying Lemma \ref{lemma_topvdcIP} we conclude that $\{n^2\}$ is a family of IP topological recurrence as
  desired.
\end{proof}
A combinatorial corollary of Theorem \ref{thm_sarkozy} states that any piecewise syndetic set $A\subset\N$ contains a pair of the form $\{x,x+n^2\}$ (this result also follows from a stronger statement first obtained independently by Furstenberg \cite{Furstenberg77} and S\'ark\"ozy \cite{Sarkozy78}, which involves sets of positive upper density in $\N$).
In addition, $n$ can be chosen from any prescribed IP-set.

  A finitistic analog of IP-sets are IP$_s$-sets, defined as follows: let $s\in\N$ and denote by ${\mathcal F}_s$ the family of non-empty subsets of $\{1,\dots,s\}$.
  Given an increasing sequence $(n_k)_{k=1}^s$ of natural numbers, define the map $n:{\mathcal F}_s\to\N$ by the formula $n_\alpha=\sum_{k\in\alpha}n_k$.
  The \emph{IP$_s$-set generated by $(n_k)$} is the set $\{n_\alpha:\alpha\in{\mathcal F}_s\}$.
  A set $A\subset\N$ that intersects non-trivially any IP$_s$ set is called an \emph{IP$_s^*$ set}.
  There exists an analog of Hindman's theorem, stating that the intersection of two IP$_s^*$-sets contains an IP$_r^*$-subset, for some $r$ which depends only on $s$ (and in particular the intersection is non-empty. For a proof of this fact see, for instance, \cite{Bergelson_Robertson14}).

  Using the PET induction scheme and applying Lemma \ref{lemma_topvdcIP} repeatedly, one can prove the general IP
  polynomial van der Waerden theorem, i.e., the fact that any family $\{a_1(n),\dots,a_k(n)\}$ where
  each $a_i\in\Z[x]$ and $a_i(0)=0$ is a family of multiple IP topological
  recurrence (see \cite{Bergelson_Leibman96}).
  One actually has the following ostensibly stronger theorem (see \cite{Bergelson_Leibman96} and \cite[Theorem 7]{Bergelson_Leibman_Ziegler11}).
\begin{theorem}\label{thm_ipsstarpolvdW}
  Let $X$ be a compact Hausdorff space, let $k\in\N$, let $T_1,\dots, T_k:X\to X$ be commuting minimal homeomorphisms, and let $U\subset X$ be a non-empty open set.
  Then for any $p_1,\dots,p_k\in\Z[x]$ with $p_i(0)=0$ there exists $s\in\N$ such that the set
  $$\left\{n\in\N:U\cap T_1^{-p_1(n)}U\cap \cdots\cap T_k^{-p_k(n)}U\neq\emptyset\right\}$$
  is IP$_s^*$
\end{theorem}

On the other hand, even the family consisting of the single function
$a(n)=\lfloor n^\alpha\rfloor$ where $\alpha\in(0,1)$ is \emph{not} a family of multiple IP topological
recurrence.
In order to study recurrence along such sequences we will make use of the notion of central$^*$ sets in a semigroup.
We refer the reader to \cite[Definition 3.1 (b)]{Bergelson_Hindman90} for a definition of central sets.
A subset of a semigroup is called a central$^*$ set if it has non-empty intersection with every central set.
The only properties of central$^*$ sets that we will need are listed in the following proposition.

\begin{proposition}\

  \begin{enumerate}
    \item If $A,B$ are central$^*$, then so is the intersection $A\cap B$.
    \item If $\phi:S\to R$ is a semigroup isomorphism and $C\subset S$ is central$^*$, then so is $\phi(C)\subset R$.
    \item \cite[Theorem 3.5]{Bergelson_Hindman94} If $A\subset\N$ is an IP$_s^*$ set, then $A$ is central$^*$ in the semigroup $(\N,\times)$.
\end{enumerate}
\end{proposition}
We prove below the following theorem:

\begin{theorem}\label{thm_centralstarn^m}
  Let $k,m\in\N$, let $X$ be a compact Hausdorff space, let $k\in\N$, let $T_1,\dots, T_k:X\to X$ be commuting minimal homeomorphisms, let $U\subset X$ be a nonempty open set and, for each $i\in\{1,\dots,k\}$ let $a_i\in\Z[x]$ with $a_i(0)=0$ and let $b_i(n)=\left\lfloor a_i\Big(n^{1/m}\Big)\right\rfloor$.
  Then the set
  $$\left\{n:U\cap T_1^{-b_1(n)}U\cap\cdots\cap T_k^{-b_k(n)}U\neq\emptyset\right\}$$
  is central$^*$ in the semigroup of $m$-th powers, $\big(\{n^m:n\in\N\},\times\big)$.
\end{theorem}
\begin{corollary}\label{cor_sqrtrectimesintersection}
  Let $k,t,m\in\N$. 
  For each $j\in\{1,\dots,t\}$ let $X_j$ be a compact Hausdorff space, let $T_{j,1},\dots, T_{j,k}:X_j\to X_j$ be commuting minimal homeomorphisms, let $U_j\subset X_j$ be open and, for each $i\in\{1,\dots,k\}$ let $a_i^{(j)}\in\Z[x]$ with $a_i^{(j)}(0)=0$ and let $b_i^{(j)}(n)=\left\lfloor a_i^{(j)}\Big(n^{1/m}\Big)\right\rfloor$.
  Then there exists $n\in\N$ such that for each $j\in\{1,\dots,t\}$ we have
  $$U_j\cap T_{j,1}^{-b_1^{(j)}(n)}U_j\cap\cdots\cap T_{j,k}^{-b_k^{(j)}(n)}U_j\neq\emptyset$$
\end{corollary}
Here is a combinatorial application of Corollary \ref{cor_sqrtrectimesintersection}.
\begin{corollary}\label{cor_lastone}
  Let $k,t,m\in\N$. 
  For each $j\in\{1,\dots,t\}$ let $A_j\subset\N^k$ be a piecewise syndetic set and, for each $i\in\{1,\dots,k\}$ let $a_i^{(j)}\in\Z[x]$ with $a_i^{(j)}(0)=0$ and let $b_i^{(j)}(n)=\left\lfloor a_i^{(j)}\Big(n^{1/m}\Big)\right\rfloor$.
  Then there exists $n\in\N$ such that for each $j\in\{1,\dots,t\}$ there exists $x_j\in\N^k$ such that
  $$\Big\{x_j+b_i^{(j)}(n)\cdot e_i:i\in\{1,\dots,k\}\Big\}\subset A_j$$
  where $\{e_1,\dots,e_k\}$ form the canonical basis for $\N^d$.
\end{corollary}
\begin{proof}[Proof of Theorem \ref{thm_centralstarn^m}]
Let
$$R_a:=\left\{n:\bigcap_{i=0}^k T^{-a_i(n)}U\neq\emptyset\right\}\quad\text{and}\quad R_b:=\left\{n:\bigcap_{i=0}^k T^{-b_i(n)}U\neq\emptyset\right\}$$
It follows from Theorem \ref{thm_ipsstarpolvdW} that $R_a$ is an IP$_s^*$ set for some $s\in\N$.
Hence, by \cite[Theorem 3.5]{Bergelson_Hindman94} we deduce that $R_a$ is central$^*$ in the semigroup $(\N,\times)$.
Let $R=\{n^m:n\in R_a\}$.

It is clear that $R\subset R_b$.
It thus suffices to show that $R$ is a central$^*$ subset of the semigroup $\big(\{n^m:n\in\N\},\times\big)$.
Indeed, the map $n\mapsto n^m$ is a semigroup isomorphism between $(\N,\times)$ and $\big(\{n^m:n\in\N\},\times\big)$, and hence the image of any central$^*$ set is again a central$^*$ set in the image semigroup.
\end{proof}

We conclude this section (and the paper) with a concrete example illustrating Corollary \ref{cor_lastone}.
\begin{example}
  For any piecewise syndetic 
  sets $A,B\subset\N$ there exist $x,y,n\in\N$ such that
  $$\big\{x,x+\lfloor \sqrt{n}\rfloor,x+\lfloor n^{7/2}\rfloor\big\}\subset A\qquad\big\{y,y+\lfloor n^{5/2}\rfloor,y+\lfloor n^{7/2}\rfloor\big\}\subset B$$
\end{example}

\paragraph{\textbf{Acknowledgements}}
The authors thank Donald Robertson for helpful comments regarding an earlier draft of this paper.

\bibliography{refs-joel}
\bibliographystyle{plain}

\end{document}